\pdfoutput=1

\documentclass[11pt]{article}

\usepackage{fullpage}

\usepackage{amsmath}
\usepackage{amssymb}
\usepackage{amsthm}
\usepackage{bm}
\usepackage{graphicx}
\usepackage{xcolor}
\usepackage{url}
\usepackage[ruled,vlined,linesnumbered,algo2e]{algorithm2e}
\usepackage{float}
\usepackage{tikz}
\usetikzlibrary{shapes.geometric}
\usepackage{enumitem}
\setlist[itemize]{leftmargin=*}
\setlist[enumerate]{leftmargin=*}
\usepackage{authblk}

\usepackage[colorinlistoftodos]{todonotes}
\newif\ifshownotes \shownotestrue
\ifshownotes
  \newcommand{\jlnote}[1]{\todo[color=blue!50,inline]{#1}}
  \newcommand{\mmnote}[1]{\todo[color=green!25,inline]{#1}}
  \newcommand{\lznote}[1]{\todo[color=magenta!15,inline]{#1}}
\else
  \newcommand{\jlnote}[1]{}
  \newcommand{\mmnote}[1]{}
  \newcommand{\lznote}[1]{}
\fi

\usepackage{hyperref}
\hypersetup{colorlinks=true, linkcolor=blue!80!black, urlcolor=blue!80!black, citecolor=blue!80!black}
\usepackage[capitalize,nameinlink]{cleveref}

\newtheoremstyle{siamplain}%
  {1ex plus .25ex minus .1ex}{1ex plus .25ex minus .1ex}%
  {\itshape}{}{\scshape}{.}{ }{}
\newtheoremstyle{siamremark}%
  {1ex plus .25ex minus .1ex}{1ex plus .25ex minus .1ex}%
  {\normalfont}{}{\itshape}{.}{ }{}

\theoremstyle{siamplain}
\newtheorem{theorem}{Theorem}
\newtheorem{lemma}[theorem]{Lemma}
\newtheorem{corollary}[theorem]{Corollary}
\newtheorem{proposition}[theorem]{Proposition}

\theoremstyle{siamremark}
\newtheorem{assumption}[theorem]{Assumption}

\crefname{assumption}{Assumption}{Assumptions}
\Crefname{assumption}{Assumption}{Assumptions}

\crefname{AlgoLine}{line}{lines}
\Crefname{AlgoLine}{Line}{Lines}
\crefname{algocf}{algorithm}{algorithms}
\Crefname{algocf}{Algorithm}{Algorithms}

\newcommand{\Ak}{\mathcal{A}_k}
\newcommand{\xb}{\mathbf{x}}
\newcommand{\yb}{\mathbf{y}}
\newcommand{\stepb}{\mathbf{s}}

\newcommand{\db}{\mathbf{d}}
\newcommand{\gammad}{\gamma_{\mathrm{dec}}}
\newcommand{\gammai}{\gamma_{\mathrm{inc}}}
\newcommand{\solver}[1]{\textsf{#1}}
\newcommand{\Algline}[2]{\hyperref[#2]{\Cref*{#1}, line~\ref*{#2}}}
\newcommand{\email}[1]{\href{mailto:#1}{#1}}

\newcommand{\TheTitle}{A Feasible Direction Method with Active-Set Identification for Derivative-Free Optimization}

\newcommand{\TheAuthors}{J. Larson, M. Menickelly, and L. Zebiane}

\title{\TheTitle\thanks{This material is based upon work supported by the U.S.~Department of
Energy, Office of Science, Office of Advanced Scientific Computing Research,
Scientific Discovery through Advanced Computing (SciDAC) Program through the
FASTMath Institute.}}

\author[1]{Jeffrey Larson}
\author[1]{Matt Menickelly}
\author[2]{Lara Zebiane}
\affil[1]{Mathematics and Computer Science Division, Argonne National Laboratory, Lemont, IL (\email{jmlarson@anl.gov}, \email{mmenickelly@anl.gov})}
\affil[2]{Department of Industrial and Systems Engineering, Lehigh University, Bethlehem, PA (\email{lara.zebiane@lehigh.edu})}

\hypersetup{pdftitle={\TheTitle}, pdfauthor={\TheAuthors}}

\begin{document}

\maketitle

\begin{abstract}
We propose a model-based trust region method for derivative-free optimization with simulation-based inequality constraints.
We are motivated by settings of simulation optimization in which both function evaluations at infeasible points may be unreliable and constraint evaluations are expensive.
Our method maintains feasible incumbents and computes steps from a subproblem suitable for a feasible-direction method in order to encourage generation of feasible trial points.
To reduce the computational cost of modeling the constraints, we describe an active-set estimation scheme that permits modeling with only a subset of constraints in each subproblem, as opposed to maintaining models of all constraints. 
We show that every limit point of the incumbents generated by our method is a Fritz-John point, and is additionally a KKT point under the Mangasarian-Fromovitz constraint qualification. 
Under standard nondegeneracy conditions, the active-set estimate identifies, within finite time, all constraints active at the limit point
of the incumbents.
Experiments on CUTEst problems show that our method generates feasible trial points more often than a comparable SQP-based approach and 
is competitive with an appropriate comparator in budgeted settings. 
\end{abstract}



\section{Introduction}\label{sec.intro}
We consider the inequality-constrained optimization problem
\begin{equation}\label{eq:opt_prob}
\begin{array}{rl}
\displaystyle\min_{\xb\in\mathbb{R}^n} & f(\xb)\\
\text{subject to} & c_i(\xb)\leq 0 \quad i=1,\dots,p,\\
\end{array}
\end{equation}
in which the objective function $f$ and each constraint function $c_i$ is assumed sufficiently smooth. 
We are additionally concerned with the \emph{derivative-free} setting, in which we have access only to function values $f(\xb)$ and $c_i(\xb)$ for a given $\xb$, and no (higher-order) derivative information. For $\xb \in \mathbb{R}^n$, we denote the set of active constraints at $\xb$ by $\mathcal{A}(\xb) := \{ i \in \{1,\dots,p\} : c_i(\xb) = 0 \}$.

Problem \cref{eq:opt_prob} has been extensively studied in the nonlinear optimization literature, leading to the development of numerous algorithmic frameworks. Among the most established and widely used are penalty and augmented Lagrangian (AL) methods \cite{ConnGoulToin91,NoceWrig06}, sequential quadratic programming (SQP) methods \cite{BoggToll95,GillMurrSaun02}, and interior-point methods \cite{WachBieg06}.

Model-based methods are widely used in derivative-free optimization (DFO); see \cite[Section 2.2]{LMW2019} for a survey. These methods construct surrogate models of the objective function $f$ (and, when applicable, the constraint functions $c_i$) using function evaluations at sampled points, typically through interpolation or regression. Derivatives of these surrogate models then serve as substitutes for the unavailable derivatives of the original functions. This approach has enabled the development of derivative-free analogues of classical nonlinear optimization methods, including penalty and augmented Lagrangian methods (e.g., \cite{AudeLeDiPeyr16,GramGrayLeDi16}) and sequential quadratic programming (SQP) methods (e.g., \cite{Trol16}).

Given the success of AL and SQP methods in nonlinear optimization, it is natural for the DFO community to adapt these methods into model-based variants. 
However, two peculiarities that arise in certain settings (often encountered in DFO) challenge some of the underlying assumptions of these nonlinear optimization methods.

The first peculiarity concerns \emph{unrelaxable constraints}, following the
taxonomy of black-box constraints in \cite{le2023taxonomy}. A constraint $c_i$
in \cref{eq:opt_prob} is unrelaxable provided points violating $c_i(\xb)\leq 0$
cannot be meaningfully evaluated.
For simulation-based constraints, such violations may correspond to inputs for
which the underlying simulation fails to produce a meaningful or interpretable objective value.
We elaborate on the specific scope of this failure mode, as it will matter in \Cref{sec.experiments}.
Throughout this paper, we assume that the constraint values $c_i(\xb)$ are well-defined at any queried point $\xb$, regardless of whether $c_i(\xb)\leq 0$, that is, regardless of whether $\xb$ is feasible with respect to $c_i$.
On the other hand, we suppose that $f(\xb)$ evaluated at any point violating a constraint may not be a reliable evaluation.
Consequently, any DFO method that permits infeasible incumbents may be unsuitable in this setting,
as they risk constructing surrogate models of the objective from invalid or
undefined values.
Although many established nonlinear optimization methods, including conventional
SQP methods, permit infeasible trial points, our proposed method maintains
feasible incumbents and guarantees that all accepted
trial points satisfy the constraints. 
This design choice makes our method more easily extendible to problems with unrelaxable constraints.
We do not offer a full extension of our method to problems with unrelaxable constraints in this manuscript, but we identify paths for near-term research that enable the treatment of unrelaxable constraints. 

As for the second peculiarity, in many practical settings of DFO, the relevant metric of utility of a given method is not iteration complexity, but is instead oracle complexity. 
In the context of a nonlinear optimization method for solving \cref{eq:opt_prob}, approximating a full constraint Jacobian by maintaining models of all $p$ constraint functions can be an expensive proposition. 
Naively, this expense could be understood by observing that a forward finite-difference approximation of the Jacobian will require $(n+1)$ evaluations of each of $p$ many constraints. 
In a simple case where all $p$ constraints are computationally independent and equally expensive, a single Jacobian approximation will entail an oracle cost of $p(n+1)$ units. 
We develop our method assuming a DFO method appropriate for expensive settings should be aware of the high expense of Jacobian approximation. 

\section{Our Contribution}
We propose a model-based trust-region (MBTR) method to address these issues.
As with any MBTR method \cite{ConnScheVice09}, 
our method will generate a sequence of incumbent points $\{\xb_k\}$ and corresponding trust-region radii $\{\Delta_k\}$. 
Our method will maintain models
$m^f \approx f$ and $m^{c_i} \approx c_i$ that are fully linear on a given trust region
$\mathcal{B}(\xb_k,\Delta_k) := \{\xb: \|\xb-\xb_k\|\leq\Delta_k\}$ where the norm $\|\cdot\|$ can be arbitrary.
We will define ``fully linear" more rigorously in \Cref{assm:interpolation}, but this term effectively means that the model makes approximation errors on $\mathcal{B}(\xb_k,\Delta_k)$ commensurate with those that would be made by a first-order Taylor model. 
Uncommon with many MBTR methods for \cref{eq:opt_prob}, the entire sequence $\{\xb_k\}$ of our method will consist of feasible points. 

\subsection{Encouraging Feasible Iterates}
To target the first issue, we will employ a subproblem that is nonstandard compared to modern approaches. 
In particular, we will employ the subproblem from the Topkis--Veinott method~\cite{TopkVein67,birge2000variant}, a feasible direction method (FDM) in the sense of Zoutendijk~\cite{Zout60}.  We refer to it as the \emph{FDM subproblem} and to the resulting step as the \emph{FDM step}.

Standard trust-region SQP methods employ a subproblem resembling
\begin{equation}\label{eq:ideal_sqp}
\begin{array}{rl}
\displaystyle\min_{\xb\in\mathcal{B}(\xb_k,\Delta_k)}
& m^f(\xb)\\
\text{subject to} & m^{c_i}(\xb)\leq 0 \quad i=1,\dots,p,\\
\end{array}
\end{equation}
where 
$\mathcal{B}(\xb_k,\Delta_k)$
denotes the trust region: a ball centered around an incumbent $\xb_k$ with radius $\Delta_k$.
In standard SQP methods, $m^f$ is a quadratic model and $m^{c_i}$ are all linear models; if one employs the $\|\cdot\|_{\infty}$ norm to define the trust region, then \cref{eq:ideal_sqp} is a quadratic program (QP), a problem for which efficient and robust solvers exist. 

In contrast (and motivated by our desire to ensure that solutions to subproblems are feasible) we will instead adopt a subproblem resembling
\begin{align}\label{eq:tv_subproblem}
    \min_{(\xb, \tau) \in \mathcal{B}(\xb_k,\Delta_k)\times \mathbb{R}}
    & \tau + \frac{1}{2}(\xb-\xb_k)^T H_k (\xb-\xb_k) \\ \nonumber
    \text{subject to} \quad
    & m^f(\xb) \leq f(\xb_k)+\tau,\\\nonumber
    &m^{c_i}(\xb) \leq \tau
    \quad  i =1,\dots,p . 
\end{align}
In \cref{eq:tv_subproblem}, $m^f$ and $m^{c_i}$ are both chosen to be linear models, and so \cref{eq:tv_subproblem} is still a QP when we employ the $\|\cdot\|_{\infty}$ norm to define the trust region. 
In contrast to \cref{eq:ideal_sqp}, the subproblem in \cref{eq:tv_subproblem} aggregates both constraint violations and objective improvement into the scalar $\tau$.
This aggregation encourages the trial step generated from the solution to \cref{eq:tv_subproblem} to be interior to the model constraint; observe that $(\xb_k, 0)$ is always feasible for \cref{eq:tv_subproblem}.
The empirical upshot of employing \cref{eq:tv_subproblem}, which we will demonstrate in this paper, is that this results more often than not in the generation of trial steps that are simultaneously feasible for all constraints. 
\Cref{fig.intro}(a) illustrates the difference between employing these two subproblems.
\begin{figure}[t]
  \centering
  \begin{minipage}[b]{0.49\linewidth}
    \centering
    \begin{tikzpicture}[scale=3.6, line width=1pt, >=stealth,
  every node/.style={font=\small}]
  \begin{scope}
    \clip (-0.80,-0.85) rectangle (0.85,0.72);
    \begin{scope}
      \clip (-0.35,-0.85) circle (1);
      \fill[black!12] (-0.75,0.8) -- (0.5,0.8) -- (0.5,-0.85) -- (-0.4375,-0.85) -- (-0.75,-0.6) -- cycle;
    \end{scope}
    \draw[black] (-0.35,-0.85) circle (1);                                   %
    \draw[black] (0.5,-1) -- (0.5,1);                                        %
    \draw[black] (-0.75,-1) -- (-0.75,1);                                  %
    \draw[black] (-1,-0.4) -- (0.5,-1.6);                                    %
    \draw[black!55] (-0.4,-0.4) rectangle (0.4,0.4);
    \draw[dashed] (-0.80,{(0.155+0.56)/1.7}) -- (0.85,{(0.155-0.595)/1.7});
    \draw[->, shorten >=2.5pt] (0,0) -- (0.4,-0.0735);
    \draw[->, shorten >=2.5pt] (0,0) -- (0.2732,-0.1547);
    \draw[->] (-0.35,-0.75) -- (-0.20,-0.705) node[right] {$-\nabla f$};
  \end{scope}
  \fill (0,0) circle (0.015) node[above left=-1pt] {$\mathbf{x}_k$};
  \fill (0.5,-0.3232) circle (0.015);
  \node[below right=0pt] at (0.5,-0.3232) {$\bar{\mathbf{x}}$};
  \node at (0.72,-0.73) {$c_1$};
  \node at (0.57,0.69) {$c_2$};
  \node at (-0.68,0.69) {$c_3$};
  \node at (-0.62,-0.80) {$c_4$};
  \draw[fill=white, line width=0.8pt] (0.4,-0.0735) circle (0.022);
  \fill (0.2732,-0.1547) circle (0.022);
  \node[anchor=south] at (0.40,-0.045) {SQP};
  \node[anchor=north] at (0.25,-0.19) {FDM};
  \node at (0.74,-0.12) {$m^{c_1}$};
\end{tikzpicture}\\[2pt]
    {\small (a) SQP versus FDM step}
  \end{minipage}\hfill
  \begin{minipage}[b]{0.49\linewidth}
    \centering
    \begin{tikzpicture}[scale=3.6, line width=1pt, >=stealth,
  every node/.style={font=\small}]
  \begin{scope}
    \clip (-0.80,-0.85) rectangle (0.85,0.72);
    \begin{scope}
      \clip (-0.35,-0.85) circle (1);
      \fill[black!12] (-0.75,0.8) -- (0.5,0.8) -- (0.5,-0.85) -- (-0.4375,-0.85) -- (-0.75,-0.6) -- cycle;
    \end{scope}
    \draw[black] (-0.35,-0.85) circle (1);                                   %
    \draw[black] (0.5,-1) -- (0.5,1);                                        %
    \draw[black!35] (-0.75,-1) -- (-0.75,1);                                  %
    \draw[black!35] (-1,-0.4) -- (0.5,-1.6);                                    %
    \draw[black!55] (-0.4,-0.4) rectangle (0.4,0.4);
    \draw[dotted] (-0.632,-0.632) rectangle (0.632,0.632);
    \draw[->] (-0.35,-0.75) -- (-0.20,-0.705) node[right] {$-\nabla f$};
  \end{scope}
  \fill (0,0) circle (0.015) node[above left=-1pt] {$\mathbf{x}_k$};
  \fill (0.5,-0.3232) circle (0.015);
  \node[below right=0pt] at (0.5,-0.3232) {$\bar{\mathbf{x}}$};
  \node at (0.72,-0.73) {$c_1$};
  \node at (0.57,0.69) {$c_2$};
  \node at (-0.68,0.69) {$c_3$};
  \node at (-0.62,-0.80) {$c_4$};
  \node[font=\footnotesize] at (0,-0.4) [below] {$\Delta_k$};
  \node[font=\footnotesize] at (-0.632,-0.60) [right] {$\Delta_k^{\sigma}$};
\end{tikzpicture}\\[2pt]
    {\small (b) Active-set estimate}
  \end{minipage}
  \caption{(a) For a curved constraint, the SQP step lands on the linearized constraint (dashed) and is infeasible; the FDM step stays feasible. (b) Near the solution $\bar{\xb}$, the trust-region radius $\Delta_k$ may become small before a step is accepted, so the trust region (gray box) can miss a constraint that is active at $\bar{\xb}$, here $c_2$. The estimate $\Ak$ is formed on the larger box of radius $\Delta_k^\sigma$ (dotted), which reaches both constraints active at $\bar{\xb}$: $\Ak = \{1,2\} = \mathcal{A}(\bar{\xb})$, while the far-away constraints $c_3$ and $c_4$ (light gray) are left out.}
  \label{fig.intro}
\end{figure}
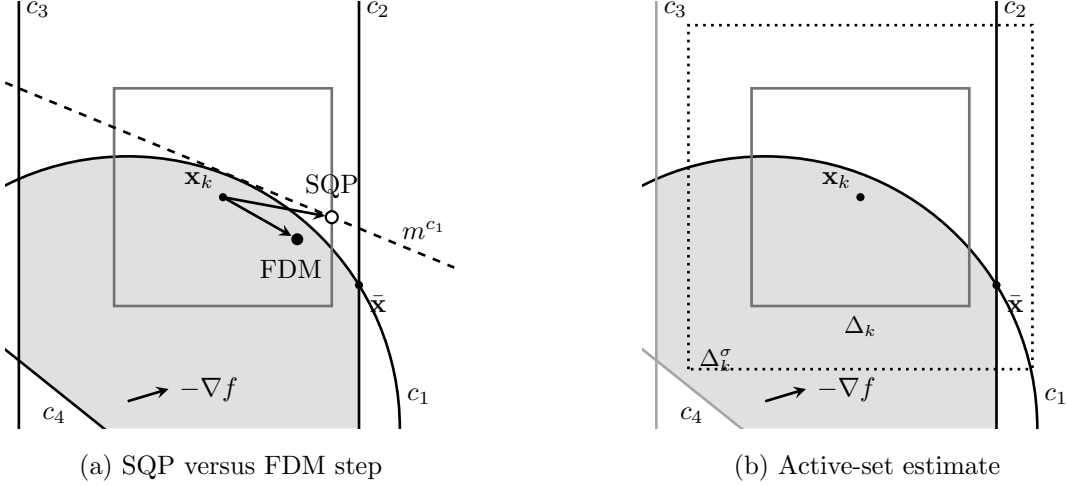

\subsection{Active-Set Estimation}
We now target the second issue of reducing function evaluations. 
Rather than modeling and enforcing every constraint at every iteration, our algorithm maintains an \emph{active-set estimate} $\Ak$ of the constraints that are likely to become active near the current trust region, and imposes only those constraints in the subproblem. Specifically, for a fixed $\sigma\in(0,1)$, we define
\begin{equation*}
  \Ak := \Big\{ i \in \{1,\dots,p\} : \max_{\xb \in \mathcal{B}(\xb_k,\Delta_k^\sigma)} m^{c_i}(\xb) \geq 0 \Big\},
\end{equation*}
which by \Cref{lem.subproblems} amounts to the test $c_i(\xb_k) + \Delta_k^{\sigma}\|g_k^{c_i}\|_1 \geq 0$. 
The test weighs the slack $c_i(\xb_k)$, which is known exactly, against the largest increase the model can achieve on a ball of radius $\Delta_k^\sigma$.
Constraints that are far from being violated in this metric can be excluded from $\Ak$, meaning function evaluations need not be spent on improving models of such constraints. 
The radius $\Delta_k^{\sigma}$ is larger than the trust-region radius $\Delta_k$ when $\Delta_k < 1$; this ensures that, near a limit point $\bar x$, the constraints active at $\bar x$ are maintained in $\Ak$, even when the trust region is too small to see the activity, see \Cref{fig.intro}(b). The effect is not on any single step but across iterations: it keeps $\Ak$ from fluctuating as $\Delta_k$ shrinks and grows near the solution, which is what makes identification possible.

Under standard nondegeneracy conditions at a limit point $\bar{\xb}$, \Cref{thm.identification} shows that, for all sufficiently large $k$, the estimate formed after each accepted step equals $\mathcal{A}(\bar{\xb})$ and is not augmented during the iteration; locally, the method behaves as if the inactive constraints were absent. 
Active-set identification has a long history in derivative-based optimization~\cite{BurkMore88,Wrig93,FaccFiscKanz98,OberWrig06} and has recently been extended to noisy and stochastic settings~\cite{CurtRobiZebi25}. 
To our knowledge, this is the first active-set identification result for a
model-based derivative-free method, and the first for derivative-free
optimization with general inequality constraints; Brilli et
al.~\cite{BrilCrisLiuzLuci26} recently established an analogous result for a
line-search-based derivative-free method restricted to bound constraints.

\Cref{sec.experiments} reports the savings in function evaluations, most pronounced under small budgets, provided by active-set estimation and lazy model updates.

\subsection{Theory and Experiments}

Our main theoretical result, \Cref{thm.main}, establishes global convergence to first-order stationarity: either the algorithm reaches a Fritz-John point of~\eqref{eq:opt_prob} in finitely many iterations, or it generates infinitely many nonzero accepted steps and every limit point of the iterates is a Fritz-John point. Under the Mangasarian-Fromovitz constraint qualification (MFCQ), these limit points are also KKT points (\Cref{thm.kkt}) and the active-set estimate identifies the constraints active at nondegenerate limit points (\Cref{thm.identification}).
In numerical experiments on problems from the CUTEst collection, an FDM subproblem produces feasible trial points far more often than an SQP subproblem.
Moreover, not updating models of constraints not in our active-set estimate $\mathcal{A}_k$ reduces the number of constraint evaluations per point and lets an implementation of our method solve problems sooner under a small budget, with a lessened advantage under a large budget; the optimal active set is identified within budget on most problems.

The remainder of the paper is organized as follows. \Cref{sec:related} reviews related work. \Cref{sec.assumptions} states persistent assumptions that will apply throughout the paper and \Cref{sec.algorithm} presents the algorithm and a global convergence analysis to Fritz-John stationary points. 
\Cref{sec.consequences} shows that, under additional assumptions (primarily MFCQ), limit points exist and limit points are KKT stationary.
\Cref{sec.identification} proves our claims concerning active-set identification
and \Cref{sec.experiments} reports numerical experiments.

\section{Related Work}\label{sec:related}

The model-based analogues of AL and SQP methods mentioned in the introduction permit infeasible iterates. The same holds for COBYLA~\cite{powell1994cobyla}, which builds linear models of the objective and constraints within a trust region and is among the most widely used solvers for constrained DFO. COBYQA~\cite{Rago22}, an SQP extension of COBYLA, likewise permits infeasibility with respect to general nonlinear constraints.

Fewer model-based methods maintain feasibility of incumbents, as we do. CONORBIT~\cite{regis2017conorbit} builds radial basis function models of the objective and constraints within a trust region.
CONORBIT can permit a small 
amount of constraint violation in its incumbents. 
Hence, CONORBIT is also potentially extendable to address the first of our identified issues, in the special case where no constraint violation is permitted.
NOWPAC~\cite{augustin2014nowpac} maintains feasible incumbents by augmenting the constraints with an inner boundary path that shrinks as the method converges. 

In general, when constraints are convex and algebraically available, so that only the objective is a black box, feasibility can be enforced exactly; this is achieved via projections in~\cite{hough2022model}. 
Unrelaxable constraints also motivate the interior-point line-search method of~\cite{BrilLiuzLuci25}, which enforces strictly feasible incumbents by means of a logarithmic barrier.
There are also modeling approaches~\cite{Padidar2022} to handle simple
unrelaxable constraints by adapting unconstrained DFO methods rather than by
constructing models of the constraints themselves.

Trying to estimate constraints that are active near an incumbent is not a new idea.
Gratton, Toint, and Tr{\"o}ltzsch~\cite{GratToinTrol11} propose a model-based
trust-region method for bound-constrained DFO that restricts the construction
of fully linear models to the subspace defined by the constraints estimated to
be active, and CONORBIT's margin~\cite{regis2017conorbit} is similarly
tightened only around constraints nearing activity. Neither establishes that
the restricted set converges to the true active set at a limit point. 
Brilli,
Cristofari, Liuzzi, and Lucidi~\cite{BrilCrisLiuzLuci26} recently proved such
an identification result, but for a line-search-based (model-free) method
restricted to bound constraints, under a strict-complementarity condition
similar to our \Cref{assm:nondegeneracy}.

Our method differs from all of these approaches in two respects. First, no margin from
the constraints is imposed or tuned as in interior point methods, CONORBIT, or NOWPAC.
Instead, the subproblem~\eqref{eq:step_subproblem}
steers trial steps away from the linearized constraints when this is
useful, while incumbents may approach, or lie on, the boundary of the feasible
region, where solutions often lie.
Second, through an active-set estimate,
the method limits the number of constraints that must be modeled and evaluated,
and we prove that this estimate identifies the constraints active at a solution.
This, to our knowledge, is the first active-set identification guarantee for
model-based DFO and, more generally, for DFO with  simulation-based inequality
constraints.

\section{Assumptions}\label{sec.assumptions}
In this section, we state assumptions that will hold throughout most of our analysis. 
Additional assumptions will be made in \Cref{sec.consequences} and \Cref{sec.identification}, but will not be immediately required in \Cref{sec.algorithm}. 
We first collect two fairly standard assumptions about the functions $f$ and $c_i$, concerning their differentiability and boundedness, respectively. 

\begin{assumption}
\label[assumption]{assm:smoothness}
  The functions $f$ and $c_i$, $i \in \{1,\dots,p\}$, are continuously differentiable with bounded and Lipschitz continuous gradients: there exist $\rho_{\nabla f}, L_{\nabla f} \in (0,\infty)$ such that $\|\nabla f(\xb)\|_2 \leq \rho_{\nabla f}$ and $\|\nabla f(\xb) - \nabla f(\bar{\xb})\|_2 \leq L_{\nabla f}\|\xb - \bar{\xb}\|_2$ for all $(\xb,\bar{\xb}) \in \mathbb{R}^n \times \mathbb{R}^n$, and likewise for each $c_i$ with constants $\rho_{\nabla c_i}, L_{\nabla c_i} \in (0,\infty)$.
\end{assumption}

\begin{assumption}\label[assumption]{assm:bounded_below}
  There exists $\underline f\in\mathbb R$ such that $f(\xb)\geq\underline f$ for all feasible $\xb$.
\end{assumption}

Our next assumption states that (somehow, to be discussed later), the models $m^f$ and $\{m^{c_i}:i=1,\dots,p\}$ are always fully linear models of $f$ and $\{c_i:i=1,\dots,p\}$ on a given sequence of trust regions. 

\begin{assumption}
\label[assumption]{assm:interpolation}
    Given a sequence of trust regions $\mathcal{B}(\xb_k,\Delta_k)$, the sequence of {linear} models $\{m_k^f\}$ and $\{m_k^{c_i}:i \in \{1,\dots,p\}\}$, are fully linear models of $f$ and $\{c_i:i=1,\dots,p\}$, respectively, on $\mathcal{B}(\xb_k,\Delta_k)$.
    That is, there exist $\kappa_f, \kappa_{\nabla f} \in (0,\infty)$ such that $|f(\xb) - m_k^f(\xb)| \leq \kappa_f \Delta_k^2$ and $\|\nabla f(\xb) - \nabla m_k^f(\xb)\|_2 \leq \kappa_{\nabla f} \Delta_k$ for all $\xb \in \mathcal{B}(\xb_k,\Delta_k)$.
    The definition holds similarly for each $c_i$ with constants $\kappa_{c_i}, \kappa_{\nabla c_i} \in (0,\infty)$.
\end{assumption}

Our last assumption in this section places restrictions on our choice of quadratic term appearing in an FDM subproblem. 
\begin{assumption}
\label[assumption]{assm:bounded_hessian}
  The sequence of matrices $\{H_k\}$ is positive definite and uniformly bounded over $k \in \mathbb{N}$, that is, there exists $\overline\zeta \in (0,\infty)$ such that, for all $k \in \mathbb{N}$, both $H_k \succ 0$ and $\stepb^T H_k \stepb \leq \overline\zeta \|\stepb\|_2^2$ for all $\stepb \in \mathbb{R}^n$.
\end{assumption}

\section{Constrained DFO Algorithm}\label{sec.algorithm}

From this point onwards, we make a consistent choice of norm in defining a trust region, that is we define $\mathcal{B}(\xb_k,\Delta) := \{\xb \in \mathbb{R}^n : \|\xb - \xb_k\|_\infty \leq \Delta\}$.
With this choice, the step subproblem~\eqref{eq:step_subproblem} is a QP. 
Our method employs a model-based criticality measure
\begin{equation}\label{eq.chik_def}
  \chi_k(\xb_k,\mathcal{A}) := -\min_{\stepb\in \mathcal{B}(0,1)} \max \left\{ {g_k^f}^T \stepb, \max_{i\in\mathcal{A}} \big( c_i(\xb_k) + {g_k^{c_i}}^T \stepb \big) \right\},
\end{equation}
where $g_k^f$ and $g_k^{c_i}$ are the model gradients defined in \Algline{alg.main_alg_chi}{step.models} and the inner maximum over an empty index set is omitted (so that $\chi_k(\xb_k,\emptyset) = \|g_k^f\|_1$).  
Its exact counterpart $\chi$, employed in our analysis, is defined in \eqref{eq.chi_def} below.

\begin{algorithm2e}[!htbp]
\caption{Feasible Active-Set Trust-Region Method}
\label{alg.main_alg_chi}
\DontPrintSemicolon
\KwIn{$\xb_0 \in \mathbb{R}^n$ with $c(\xb_0) \leq 0$;\, $\Delta_0 \geq \underline\Delta > 0$;\, $\mu > 0$;\, $\sigma, \gammad \in (0,1)$;\, $\gammai \geq 1$;\, $0 < \eta_1 \leq \eta_2 < 1$}
\For{$k = 0,1,2,\dots$}{
  Build/update linear models $m^f(\xb_k + \stepb) = f(\xb_k) + {g_k^f}^T\stepb$ and $m^{c_i}(\xb_k + \stepb) = c_i(\xb_k) + {g_k^{c_i}}^T\stepb$, $i = 1,\dots,p$\nllabel{step.models}\;
  Set
  \begin{equation}\label{eq:active_subproblem}
    \Ak \gets \Big\{ i \in \{1,\dots,p\} : \max_{\tilde{\stepb} \in \mathcal{B}(0,\Delta_k^\sigma)} m^{c_i}(\xb_k + \tilde{\stepb}) \geq 0 \Big\}
  \end{equation}\nllabel{step.active_set}\vspace{-1.1\baselineskip}\;
  \While{$\Delta_k > \mu\, \chi_k(\xb_k,\Ak)$\nllabel{step.criticality}}{
    Set $\Delta_k \gets \gammad \Delta_k$\; 
    Rebuild the models as in line~\ref{step.models} and $\Ak$ as in \eqref{eq:active_subproblem}\;
  }
  \While{true\nllabel{step.mainwhile}}{
    Choose $H_k \succ 0$ and compute\nllabel{step.sk}
    \begin{equation}\label{eq:step_subproblem}
      \begin{aligned}
        (\stepb_k,\tau_k) \in \arg\min_{\stepb \in \mathcal{B}(0,\Delta_k),\ \tau \in \mathbb{R}}\ & \tau + \tfrac12 \stepb^T H_k \stepb \\
        \text{s.t.}\ & m^f(\xb_k + \stepb) \leq f(\xb_k) + \tau, \\
                     & m^{c_i}(\xb_k + \stepb) \leq \tau\ \ \text{for all}\ i \in \Ak
      \end{aligned}
    \end{equation}\vspace{-1.1\baselineskip}\;
    Set $i_{\mathrm{viol}} \gets 0$\;
    \For{$i = 1,\dots,p$}{
      Evaluate $c_i(\xb_k + \stepb_k)$ \nllabel{line:eval_ci}
      \lIf{$c_i(\xb_k + \stepb_k) > 0$}{set $i_{\mathrm{viol}} \gets i$ and \textbf{break}\nllabel{line:set_iviol}}
    }
    \uIf{$i_{\mathrm{viol}} = 0$ \textnormal{and} $\stepb_k \neq 0$\nllabel{line:iviol_zero}}{
      Evaluate $f(\xb_k + \stepb_k)$\; 
      Compute $\rho_k = \big(f(\xb_k) - f(\xb_k + \stepb_k)\big) / \big(m^f(\xb_k) - m^f(\xb_k + \stepb_k)\big)$\;
      \lIf{$\rho_k \geq \eta_2$}{set $\xb_{k+1} \gets \xb_k + \stepb_k$ and $\Delta_{k+1} \gets \max\{\underline\Delta, \gammai\Delta_k\}$\nllabel{line:very_successful}}
      \lElseIf{$\rho_k \geq \eta_1$}{set $\xb_{k+1} \gets \xb_k + \stepb_k$ and$\Delta_{k+1} \gets \max\{\underline\Delta, \Delta_k\}$\nllabel{line:successful}}
      \lElse{set $\xb_{k+1} \gets \xb_k$ and $\Delta_{k+1} \gets \gammad\Delta_k$\nllabel{line:unsuccessful}}
      Go to iteration $k+1$\;
    }
    \uElseIf{$i_{\mathrm{viol}} \geq 1$, $i_{\mathrm{viol}} \notin \Ak$, \textnormal{and} $\stepb_k \neq 0$\nllabel{line:expand_active_set}}{
      Set $\tilde{\mathcal{A}}_k \gets \Ak \cup \{i_{\mathrm{viol}}\}$ and build/update the linear model $m^{c_{i_{\mathrm{viol}}}}$\;
      \uIf{$\Delta_k \leq \mu\, \chi_k(\xb_k,\tilde{\mathcal{A}}_k)$\nllabel{step.critical_verify}}{
        Set $\Ak \gets \tilde{\mathcal{A}}_k$\nllabel{step.critical_verify_update}\;
      }
      \lElse{set $\xb_{k+1} \gets \xb_k$ and $\Delta_{k+1} \gets \gammad\Delta_k$; go to iteration $k+1$}
    }
    \lElse{set $\xb_{k+1} \gets \xb_k$ and $\Delta_{k+1} \gets \gammad\Delta_k$; go to iteration $k+1$\nllabel{line:rejected_viol}}
  }
}
\end{algorithm2e}

\subsection{Discussion of the algorithm}

Here, we introduce and explain a few choices made in \Cref{alg.main_alg_chi}. 
\Cref{alg.main_alg_chi} starts from a feasible point $\xb_0$ and indexes iterations by $k=0,1,\dots$. 
Recall that the logic of the algorithm is going to be designed to ensure that all incumbents $\xb_k$ remain feasible as the algorithm progresses. 
At iteration $k$, the algorithm either builds or updates linear models $m^f$ and $m^{c_i}$ of, respectively, $f$ and $c_i$, around the current incumbent $\xb_k$. 
The analysis will require (viz. \Cref{assm:interpolation}) that these linear models be fully linear models. 

\Cref{alg.main_alg_chi} then selects a subset of constraints that are potentially active, according to \eqref{eq:active_subproblem}. 
As a practical matter, because the models $m^{c_i}$ are linear and trust regions are defined in the $\ell_\infty$ norm, $\Ak$ has the more explicit form
\begin{equation}\label{eq.Ak_explicit}
  \Ak = \{i\in\{1,\dots,p\}:c_i(\xb_k)+\Delta_k^\sigma\|g_k^{c_i}\|_1\geq0\}.
\end{equation}
We also note that the exponential factor $\sigma>0$ is necessary to demonstrate the identification result in \Cref{sec.identification} in the spirit of~\cite{CurtRobiZebi25,OberWrig06}. 

\Cref{alg.main_alg_chi} will then enter what we call the \emph{criticality loop} in \Algline{alg.main_alg_chi}{step.criticality}. 
The criticality loop will iteratively shrink $\Delta_k$ until it is sufficiently small compared to the criticality measure $\chi_k(\xb_k,\Ak)$.  
With each shrinkage of $\Delta_k$, we must of course reevaluate our estimate $\Ak$. 
In theory, each shrinkage of $\Delta_k$ also entails ensuring that $m^f$ and all constraint models $m^{c_i}$ satisfy \Cref{assm:interpolation} on progressively smaller trust-regions.
In practice, however, assuring \Cref{assm:interpolation} holds for all $p+1$ models can accrue large computational expense.
We find that ensuring all such approximation bounds is unnecessary and we will only enforce fully linearity of $\{m^{c_i}:i\in\mathcal{A}_k\}$, where $\Ak$ in \eqref{eq.Ak_explicit} is determined from the most recently computed $g_k^{c_i}$. 

Once the trust-region radius $\Delta_k$ is commensurate with the criticality measure $\chi_k(\xb_k,\Ak)$, we can enter the \emph{main loop} beginning in \Algline{alg.main_alg_chi}{step.mainwhile}, which will generate at least one trial point. 
\Cref{alg.main_alg_chi} will solve the step subproblem \eqref{eq:step_subproblem} using the constraints in $\Ak$.
 Recalling that because $(\stepb,\tau)=(0,0)$ is feasible for any given problem data, any nonzero minimizing step has $\tau_k<0$. 
 Such a step with $\tau_k<0$ therefore predicts objective decrease and is strictly interior to the model constraints indexed by $\Ak$. 
 As we will show in \Cref{lem.well-defined}, if the criticality loop was exited, then there exists such a nonzero step, and hence, a positive predicted objective decrease. 
 
 These properties of $(\stepb,\tau)$ encourage feasibility of the trial point, but cannot guarantee it.
Thus, \Cref{alg.main_alg_chi} evaluates all of $\{c_i(\xb_k+\stepb_k): i=1,\dots,p\}$, one at a time, stopping at the first constraint violation.
If the constraint $c_i$ is violated and both $i\not\in\Ak$ and $\stepb_k\neq 0$, then 
the model $m^{c_i}$ is updated (so as to satisfy \Cref{assm:interpolation}) and $i$ is tentatively added to $\Ak$. 
Since growing $\Ak$ can generally decrease the criticality measure $\chi_k(\xb_k,\Ak)$, \Cref{alg.main_alg_chi} must recheck the criticality condition. 
If the criticality condition still holds with the tentative $\Ak$, we let the tentative $\Ak$ become the value of $\Ak$;
if the criticality condition would fail with the tentative $\Ak$, we shrink the trust region radius and start a new iteration. 
 
If, on the other hand, all constraints are satisfied ($c_i(\xb_k+\stepb_k)\leq 0$ for all $i=1,\dots,p$) and $\stepb_k$ is nonzero, then \Cref{alg.main_alg_chi} finally evaluates the objective $f(\xb_k+\stepb_k)$ and computes $\rho_k$, the ratio of actual to predicted objective decrease. 
As in standard trust-region methods, the trial incumbent $\xb_k+\stepb_k$ is accepted when $\rho_k\geq\eta_1$; if also $\rho_k\geq\eta_2$, the trust-region radius is increased. 
On the other hand, if $\rho_k<\eta_1$, then the incumbent $\xb_k$ is unchanged and the radius is contracted. 
Thus, accepted steps decrease the objective and all incumbents $\{\xb_k\}$ are feasible. 

Compared to standard trust-region methods, and as a peculiarity of \Cref{alg.main_alg_chi}, 
the trust-region radius is additionally reset to at least $\underline\Delta$ following any accepted step.
This reset supplies a lower bound used in the entry condition for the identification argument. 
The reset does not impose a lower bound on the radii used to compute steps in \eqref{eq:step_subproblem}; subsequent criticality loops and unsuccessful iterations may reduce the radius to a value again below $\underline\Delta$.

We also note that, in the statement of \Cref{alg.main_alg_chi}, there is no stopping criterion provided. We will demonstrate in \Cref{lem.well-defined} that if the criticality loop cannot be exited in the $k$th iteration, then $\xb_k$ is a Fritz-John point. 
Otherwise, practical stopping tests will be described in \Cref{sec.impl}. 

\subsection{Analysis}
For shorthand notation, denote $P:=\{1,\dots,p\}$. 
We say a feasible point $\xb$ is a \emph{Fritz-John point} provided there exist multipliers $(y^f,y^c)\geq0$ satisfying
\begin{equation}\label{eq.FritzJohn}
  y^f\nabla f(\xb)+\sum_{i\in P}y_i^c\nabla c_i(\xb)=0,
  \qquad y^f+\mathbf1^Ty^c=1,
  \qquad y_i^c c_i(\xb)=0\quad(i\in P).
\end{equation}
We propose an exact criticality measure 
\begin{equation}\label{eq.chi_def}
  \chi(\xb,\mathcal A):=-\min_{\stepb\in\mathcal B(0,1)}
  \max\left\{\nabla f(\xb)^T\stepb,
    \max_{i\in\mathcal A}\big(c_i(\xb)+\nabla c_i(\xb)^T\stepb\big)\right\},
\end{equation}
with the inner maximum omitted when $\mathcal A=\emptyset$. 
As further shorthand notation, we will denote the simplex of weights corresponding to a set $\mathcal{A}\subseteq P$ via 
$$\Sigma_{\mathcal A}:=\{(y^f,y^c)\geq0:y^f+\sum_{i\in\mathcal A}y_i^c=1\}.$$
We will first demonstrate that \eqref{eq.chi_def} properly defines a criticality measure, as intended. 
\begin{proposition}\label{prop.criticality}
  Suppose \Cref{assm:smoothness} holds. For any $\xb\in\mathbb{R}^n$ and $\mathcal A\subseteq P$,
  \begin{equation}\label{eq.dual_chi}
    \chi(\xb,\mathcal A)=\min_{y\in\Sigma_{\mathcal A}}
    \left\{\left\|y^f\nabla f(\xb)+\sum_{i\in\mathcal A}y_i^c\nabla c_i(\xb)\right\|_1
      -\sum_{i\in\mathcal A}y_i^c c_i(\xb)\right\}.
  \end{equation}
  The function $\chi(\cdot,\mathcal A)$ is continuous, and both $\chi$ and $\chi_k$ are nonincreasing when we take supersets of $\mathcal A$. 
  For feasible $\xb$ with $\mathcal A(\xb)\subseteq\mathcal A$, one has $\chi(\xb,\mathcal A)=0$ if and only if $\xb$ is a Fritz-John point. 
  If, in addition, \Cref{assm:interpolation} holds, then
  \begin{equation}\label{eq.chi_error}
    |\chi_k(\xb_k,\mathcal A)-\chi(\xb_k,\mathcal A)|\leq\kappa_\chi\Delta_k,
    \qquad
    \kappa_\chi:=\sqrt n\max\{\kappa_{\nabla f},\max_{i\in P}\kappa_{\nabla c_i}\}.
  \end{equation}
\end{proposition}
\begin{proof}
 We can equivalently express \eqref{eq.chi_def} as 
 \begin{equation}\label{eq.chi_rewrite}
  \chi(\xb,\mathcal A):=-\min_{\stepb\in\mathcal B(0,1)}
  \max_{\yb\in\Sigma_{\mathcal A}} 
  \left\{[0; c_{\mathcal{A}}(\xb)]^\top \yb + \stepb^\top [\nabla f(\xb);  \nabla c_{\mathcal{A}}(\xb)] \yb
 \right\}.
\end{equation}
 This reformulation of~\eqref{eq.chi_def} is affine in $\yb$ and $\stepb$, each on a respective compact convex set, and so minimax duality permits interchanging the minimization and maximization in \eqref{eq.chi_rewrite}.
  Minimizing the linear function of $\stepb$ in~\eqref{eq.chi_rewrite} over the unit ball yields~\eqref{eq.dual_chi}. 
  At feasible $\xb$, both terms in the objective of~\eqref{eq.dual_chi} are nonnegative; a zero minimum is therefore equivalent to~\eqref{eq.FritzJohn}, with multipliers extended by zero outside $\mathcal A$. Conversely, every Fritz-John multiplier is supported on $\mathcal A(\xb)\subseteq\mathcal A$.

  Continuity follows from minimization of a jointly continuous function over a fixed compact set; monotonicity follows by enlarging the inner maximum in~\eqref{eq.chi_def}. For every $\stepb\in\mathcal B(0,1)$, full linearity bounds the perturbation of each affine term by $\kappa_\chi\Delta_k$. Taking maxima and then minima preserves this uniform bound, proving~\eqref{eq.chi_error}.
\end{proof}

For proving global convergence results, it suffices to consider $\widehat\chi(\xb):=\chi(\xb,P)$. 
In particular, due to \Cref{prop.criticality}, uniformly over any $\mathcal{A}\subseteq P$, 
\begin{equation}\label{eq.full_criticality}
  \chi_k(\xb_k,\mathcal A)\geq\widehat\chi(\xb_k)-\kappa_\chi\Delta_k.
\end{equation}
On a fixed iteration $k$, the bound \eqref{eq.full_criticality} applies after any model updates (due to the conservatism of $\kappa_\chi$) and active-set augmentations (due to $\mathcal{A}\subseteq P$).
 In what follows, $\Delta_k$ and $\Ak$ denote their current values within iteration $k$; on completed iterations, $\Delta_k$ denotes the radius used in the subproblem \eqref{eq:step_subproblem}.
 
 In our next result, we show that \Cref{alg.main_alg_chi} is well-defined in the sense that no inner loop can continue indefinitely.
\begin{lemma}\label{lem.well-defined}\label{lem.cauchy}\label{lem.subproblems}
 Suppose \Cref{assm:smoothness,assm:interpolation,assm:bounded_hessian} hold. For all $k$,
 \begin{enumerate}
 \item $\xb_{k+1}$ is feasible and $\Delta_k>0$.
\item If $\xb_k$ is not a Fritz-John point, then the criticality loop beginning in \Algline{alg.main_alg_chi}{step.criticality} terminates finitely.
\item The while loop beginning in \Algline{alg.main_alg_chi}{step.mainwhile} terminates after at most $p$ passes.
\item If \Algline{alg.main_alg_chi}{step.sk} is reached, then \eqref{eq:step_subproblem} admits a unique primal solution $\stepb_k$.
\item Writing $\chi_k:=\chi_k(\xb_k,\Ak)$, if \Algline{alg.main_alg_chi}{step.sk} is reached, then $\chi_k\geq\Delta_k/\mu>0$ and $\stepb_k$ satisfies
  \begin{equation}\label{eq.starpre}
    m^f(\xb_k)-m^f(\xb_k+\stepb_k)\geq-\tau_k
    \geq\tfrac12\chi_k\min\{1,\Delta_k,\chi_k/(\overline\zeta n)\}>0.
  \end{equation}
  \end{enumerate}
\end{lemma}
\begin{proof}
The first and third claim are both immediate by the logic of the algorithm. 

  Let $\xb_k$ be a feasible incumbent. 
 If the loop beginning in \Algline{alg.main_alg_chi}{step.criticality} repeats infinitely, then $\Delta_k\to0$, and the loop condition and~\eqref{eq.full_criticality} together give
  $0\leq\widehat\chi(\xb_k)<(\mu^{-1}+\kappa_\chi)\Delta_k$.
  Thus $\xb_k$ is a Fritz-John point by \Cref{prop.criticality}, proving the second claim. 

The subproblem~\eqref{eq:step_subproblem} minimizes a strictly convex function over a compact set.
 Observing that $(\mathbf{0},0)$ is feasible for~\eqref{eq:step_subproblem}, we conclude that the minimizer $\stepb_k$ exists and is unique, proving the fourth claim. 

  Now, if the loop beginning in \Algline{alg.main_alg_chi}{step.criticality} terminates (so $\xb_k$ is not a Fritz-John point) then the loop's termination condition, and the check in \Algline{alg.main_alg_chi}{step.critical_verify}, ensure $\chi_k\geq\Delta_k/\mu>0$ every time \Algline{alg.main_alg_chi}{step.sk} is reached.
  Let $\hat\stepb$ be a unique minimum defining $\chi_k$ and set
  $t:=\min\{1,\Delta_k,\chi_k/(\overline\zeta n)\}>0$.
  Since $c_i(\xb_k)\leq0$ and $t\leq1$, the pair $(t\hat\stepb,-t\chi_k)$ is feasible for~\eqref{eq:step_subproblem}. 
  The optimality of $\stepb_k$ and the Hessian bound in \Cref{assm:bounded_hessian} together yield
  \[
    \tau_k+\tfrac12\stepb_k^TH_k\stepb_k
      \leq-t\chi_k+\tfrac12\overline\zeta n t^2
      \leq-\tfrac12t\chi_k.
  \]
  The constraints of \eqref{eq:step_subproblem} then yield~\eqref{eq.starpre}, and so $\stepb_k\neq0$ and the acceptance ratio is well defined, proving the fifth claim. 
\end{proof}

We now show that when we enter the loop beginning in \Algline{alg.main_alg_chi}{step.mainwhile} with sufficiently small $\Delta_k$, and if $\xb_k$ is not a Fritz-John point, then the trial step $\xb_k+\stepb_k$ generated at the end of the loop must be accepted.

\begin{lemma}\label{lem.local}
Suppose \Cref{assm:smoothness,assm:interpolation,assm:bounded_hessian} hold, and let $\bar\xb$ be feasible but not a Fritz-John point.
There exist $\epsilon,a,\Delta^*>0$ such that, provided $\xb_k$ satisfies $\|\xb_k-\bar\xb\|_\infty\leq\epsilon$ 
and %
$\Delta_k<\Delta^*$, 
then
 $\mathcal{A}_k$ satisfies
  $\chi_k(\xb_k,\mathcal A_k)\geq a$ and $\Delta_k\leq\mu\chi_k(\xb_k,\mathcal A_k)$.
  If the loop beginning in \Algline{alg.main_alg_chi}{step.mainwhile} is reached with $\Delta_k <\Delta^*$, then the $k$th iteration ends with an accepted step $\xb_{k+1} = \xb_k + \stepb_k$ and the predicted decrease satisfies
   $m^f(\xb_k) - m^f(\xb_{k+1}) \geq a\Delta_k/2$.
\end{lemma}
\begin{proof}
  By \Cref{prop.criticality}, there exist $a,\epsilon>0$ so that $\widehat\chi(\xb)\geq2a$ on $\mathcal B(\bar\xb,\epsilon)$. 
  Let $\Delta^*>0$ satisfy
  \[
    \Delta^*\leq\min\left\{1,\frac a{\kappa_\chi},\frac a{\overline\zeta n},
    \mu a,\frac{(1-\eta_1)a}{2\kappa_f},
    \min_{i\in P}\frac a{2\kappa_{c_i}}\right\}.
  \]
  Recall that Equation~\eqref{eq.full_criticality} holds for any $\mathcal{A}\subseteq P$.
 Thus, whenever the models are fully linear on $\mathcal{B}(\xb_k,\Delta_k)$ with $\Delta_k<\Delta^*$, we have that $\chi_k(\xb_k,\mathcal A)
        \geq\widehat\chi(\xb_k)-\kappa_\chi\Delta_k
         \geq2a-\kappa_\chi\Delta_k\geq a,$ and so 
 $\Delta_k<\Delta^*\leq\mu a\leq\mu\chi_k(\xb_k,\mathcal A).$
  Letting $\mathcal A=\Ak$, we prove the first claim. 

  Suppose now that the loop beginning in \Algline{alg.main_alg_chi}{step.mainwhile} is reached.
   Index its successive passes by $j=0,1,\dots$, writing $\Ak^{(j)}$, $(\stepb_k^{(j)},\tau_k^{(j)})$ for the working set and subproblem solution at pass $j$, respectively. 
   Both $\xb_k$ (feasible, by \Cref{lem.well-defined}) and $\Delta_k$ remain fixed during this loop.

  At $j=0$, the proven first claim gives $\chi_k(\xb_k,\Ak^{(0)})\geq a$. Since $\Delta_k<\Delta^*\leq\min\{1,a/(\overline\zeta n)\}$, the minimum in~\eqref{eq.starpre} is $\Delta_k$. Combining that inequality with the objective and constraint inequalities of~\eqref{eq:step_subproblem} yields
  \begin{equation}\label{eq.local_pass_bounds}
    \begin{aligned}
      m^f(\xb_k)-m^f(\xb_k+\stepb_k^{(j)})
        &\geq-\tau_k^{(j)}\geq\tfrac12a\Delta_k,\\
      m^{c_i}(\xb_k+\stepb_k^{(j)})
        &\leq\tau_k^{(j)}\leq-\tfrac12a\Delta_k
        \qquad(i\in\Ak^{(j)}),
    \end{aligned}
  \end{equation}
  initially for $j=0$. In particular, $\stepb_k^{(0)}\neq0$.

  Suppose \eqref{eq.local_pass_bounds} held at pass $j$; as an induction, we will prove that \eqref{eq.local_pass_bounds} holds at the $(j+1)$th pass.
   Since $\|\stepb_k^{(j)}\|_\infty\leq\Delta_k$, full linearity (\Cref{assm:interpolation}) gives, for every $i\in\Ak^{(j)}$,
  \begin{equation}\label{eq:feas_bound}
    c_i(\xb_k+\stepb_k^{(j)})
    \leq m^{c_i}(\xb_k+\stepb_k^{(j)})+\kappa_{c_i}\Delta_k^2
    \leq-\tfrac12a\Delta_k+\kappa_{c_i}\Delta_k^2\leq0,
  \end{equation}
  where the last inequality uses $\Delta_k<\Delta^*\leq a/(2\kappa_{c_i})$.
  Consequently, any violated index $i_{\mathrm{viol}}$ in \Algline{alg.main_alg_chi}{line:set_iviol} must belong to $P\setminus\Ak^{(j)}$. 
  Since~\eqref{eq.local_pass_bounds} also implies $\stepb_k^{(j)}\neq0$, the algorithm will then set 
   $\widetilde{\mathcal A}_k=\Ak^{(j)}\cup\{i_{\mathrm{viol}}\}$.
  After the model of $c_{i_{\mathrm{viol}}}$ is built or updated, it is fully linear on $\mathcal{B}(\xb_k,\Delta_k)$. 
  Applying~\eqref{eq.full_criticality} to the updated models therefore gives
  \[
    \begin{aligned}
      \chi_k(\xb_k,\widetilde{\mathcal A}_k)
        &\geq\widehat\chi(\xb_k)-\kappa_\chi\Delta_k\geq a,\\
      \Delta_k&<\Delta^*\leq\mu a
        \leq\mu\chi_k(\xb_k,\widetilde{\mathcal A}_k).
    \end{aligned}
  \]
  Thus the check in \Algline{alg.main_alg_chi}{step.critical_verify} passes and the algorithm sets $\Ak^{(j+1)}=\widetilde{\mathcal A}_k$. 
Thus,~\eqref{eq.local_pass_bounds} holds in the $(j+1)$th pass.

By \Cref{lem.well-defined}, there can be at most $p-|\Ak^{(0)}|$ many passes through the loop beginning in \Algline{alg.main_alg_chi}{step.mainwhile}.
The bound \eqref{eq:feas_bound} applies to every constraint; hence some pass $j_*$ must produce a feasible trial point. 
After the $j_*$th pass, we have that $\stepb_k=\stepb_k^{(j_*)}$. 
Using $m^f(\xb_k)=f(\xb_k)$, full linearity of the objective model, and~\eqref{eq.local_pass_bounds}, we obtain
  \[
    |\rho_k-1|
      =\frac{|f(\xb_k+\stepb_k)-m^f(\xb_k+\stepb_k)|}
             {m^f(\xb_k)-m^f(\xb_k+\stepb_k)}
      \leq\frac{\kappa_f\Delta_k^2}{a\Delta_k/2}
      =\frac{2\kappa_f\Delta_k}{a}\leq1-\eta_1.
  \]
  Thus $\rho_k\geq\eta_1$, so this nonzero step is accepted and its predicted decrease is at least $a\Delta_k/2$, as claimed.
\end{proof}

\begin{theorem}[Global convergence]\label{thm.main}
  Under \Cref{assm:smoothness,assm:interpolation,assm:bounded_hessian,assm:bounded_below}, either some $\xb_k$ is a Fritz-John point, or else \Cref{alg.main_alg_chi} generates infinitely many nonzero accepted steps and every limit point of $\{\xb_k\}$ is a Fritz-John point.
\end{theorem}
\begin{proof}
  Suppose no $\xb_k$ is a Fritz-John point.
  Infinitely many steps must be accepted; otherwise $\xb_k$ is eventually fixed while rejections drive $\Delta_k\to 0$, contradicting \Cref{lem.local}.
   Let $\mathcal S$ index the infinitely many iterations $k$ in which a step is accepted.
As shorthand notation, define $D_k:=m^f(\xb_k)-m^f(\xb_k+\stepb_k)$.
By definition of accepted, $f(\xb_k)-f(\xb_{k+1})\geq\eta_1D_k>0$ for $k\in\mathcal S$; $f(\xb_k) - f(\xb_{k+1}) = 0$ for $k\not\in\mathcal S$.
By \Cref{assm:bounded_below}, $f(\xb_k)\geq\underline f$ for all $k$, so
  \begin{equation}\label{eq.sum_model}
    \sum_{k\in\mathcal S}D_k\leq\frac{f(\xb_0)-\underline f}{\eta_1}<\infty,
    \qquad
    D_k\geq\frac{\Delta_k}{2\mu}
       \min\left\{1,\Delta_k,\frac{\Delta_k}{\mu\overline\zeta n}\right\},
  \end{equation}
  where the second inequality follows from~\eqref{eq.starpre} and $\chi_k\geq\Delta_k/\mu$. Summability gives $D_k\to0$ over $k\in\mathcal S$. Since the lower bound in~\eqref{eq.sum_model} is bounded away from zero whenever $\Delta_k$ is bounded away from zero, it follows that
  \begin{equation}\label{eq.Dvanish}
    \Delta_k\to0\qquad\text{as }k\to\infty\text{ over }k\in\mathcal S.
  \end{equation}

  Let $\bar\xb$ be any limit point of $\{\xb_k\}$ implied by \eqref{eq.Dvanish}, necessarily feasible, and suppose towards a contradiction that $\bar\xb$ is not a Fritz-John point. 
  Take $\epsilon,\Delta^*$ from \Cref{lem.local} and define
    $\Delta_{\rm acc}:=\min\{\underline\Delta,\gammad\Delta^*\}>0.$
  We show that every iteration with $\xb_k\in\mathcal B(\bar\xb,\epsilon)$ reaches the loop beginning in \Algline{alg.main_alg_chi}{step.mainwhile} with $\Delta_k\geq\Delta_{\rm acc}$.

  First, denote the radius on entry to the criticality loop beginning in \Algline{alg.main_alg_chi}{step.criticality} by $\Delta_k^{\rm in}$. If $k=0$ or iteration $k-1$ resulted in acceptance, initialization or the reset gives $\Delta_k^{\rm in}\geq\underline\Delta$. Otherwise, $\xb_{k-1}=\xb_k\in\mathcal B(\bar\xb,\epsilon)$ and the preceding iteration resulted in rejection.
  By \Cref{lem.local},  $\Delta_{k-1}\geq\Delta^*$, since $\Delta_{k-1}<\Delta^*$ would have resulted in acceptance. Therefore
  $\Delta_k^{\rm in}=\gammad\Delta_{k-1}\geq\gammad\Delta^*$.
  In either case, $\Delta_k^{\rm in}\geq\Delta_{\rm acc}$.

\Cref{lem.local} ensures that the criticality loop ends when $\Delta_k<\Delta^*$. 
Thus, any reduction of $\Delta_k$ performed by the criticality loop started from $\Delta_k > \Delta^*$. 
If no reduction occurs, the criticality loop ends with $\Delta_k^{\rm in}$; 
otherwise, the last reduction results in $\Delta_k\geq\gammad\Delta^*$. Consequently, the terminal radius, which is also the radius used throughout the loop beginning in \Algline{alg.main_alg_chi}{step.mainwhile}, satisfies $\Delta_k\geq\Delta_{\rm acc}$.

  Finally, let $k_\ell\to\infty$ with $\xb_{k_\ell}\to\bar\xb$, and let
  $j_\ell:=\min\{j\in\mathcal S:j\geq k_\ell\}$, which exists because $\mathcal S$ is infinite. Every iteration from $k_\ell$ through $j_\ell-1$ rejects its step and leaves the incumbent unchanged, so $\xb_{j_\ell}=\xb_{k_\ell}\to\bar\xb$. Also $j_\ell\geq k_\ell\to\infty$. Hence this subset of $\mathcal{S}$ eventually witnesses both $\xb_{j_\ell}\in\mathcal B(\bar\xb,\epsilon)$ and $\Delta_{j_\ell}\geq\Delta_{\rm acc}>0$, contradicting~\eqref{eq.Dvanish}. Thus every limit point is a Fritz-John point.
\end{proof}

\section{Consequences of Global Convergence}\label{sec.consequences}
By making slightly stronger assumptions, we can additionally demonstrate that \Cref{alg.main_alg_chi} converges to points satisfying Karush-Kuhn-Tucker (KKT) conditions.
Specifically, we say a feasible point is a \emph{KKT point} provided there exists $\mathbf{\lambda}\geq0$ satisfying
$\nabla f(\xb)+\sum_i\lambda_i\nabla c_i(\xb)=0$ and $\lambda_i c_i(\xb)=0$ for all $i$.
Every sequence $\{\xb_k\}$ generated by \Cref{alg.main_alg_chi}, and hence the limit points of any such sequences, belongs to the closed feasible level set
\begin{equation}\label{eq.levelset}
  \mathcal L:=\{\xb:c(\xb)\leq0,\ f(\xb)\leq f(\xb_0)\}.
\end{equation}
Our first assumption simply says that the level set is bounded, while the second is a standard constraint qualification. 
\begin{assumption}\label[assumption]{assm:bounded_levelset}
  The set $\mathcal L$ is bounded.
\end{assumption}
\begin{assumption}\label[assumption]{assm:mfcq}
  At every $\xb\in\mathcal L$, the Mangasarian--Fromovitz constraint qualification (MFCQ) holds: some direction $\db$ satisfies
  $\nabla c_i(\xb)^T\db<0$ for all $i\in\mathcal A(\xb)$.
\end{assumption}

We can now prove an immediate corollary to \Cref{thm.main}, the second part of which shows that $\{\xb_k\}$ generated by \Cref{alg.main_alg_chi} converges to KKT points. 

\begin{corollary}\label{cor.existence}\label{cor.connected}\label{cor.convergence}\label{thm.kkt}
  Under the assumptions of \Cref{thm.main}, the following hold.
  \begin{enumerate}
    \item[(a)] If \Cref{assm:bounded_levelset} holds and no $\xb_k$ is a Fritz-John point, then $f(\xb_k)$ converges and the limit set $\mathcal X$ of $\{\xb_k\}$ is nonempty, compact, connected, and consists entirely of Fritz-John points. If $\mathcal{X}$ contains an isolated point $\bar\xb$, then $\xb_k\to\bar\xb$, that is, $\mathcal X = \{\bar\xb\}.$
    \item[(b)] If \Cref{assm:mfcq} also holds, then either some $\xb_k$ is a KKT point, or else infinitely many nonzero steps are accepted in \Cref{alg.main_alg_chi} and every limit point of $\{\xb_k\}$ is a KKT point. 
  \end{enumerate}
\end{corollary}
\begin{proof}
  To prove (a), compactness of $\mathcal L$ and the monotonicity ensured by successful steps automatically yield the convergence of $f(\xb_k)$ and a nonempty compact set of limit points $\mathcal{X}$, the elements of which are all Fritz-John points due to \Cref{thm.main}. 
The connectedness of $\mathcal{X}$ follows since \eqref{eq.Dvanish} implies $\|\xb_{k+1}-\xb_k\|_\infty\to0$, and the set of limit points of a bounded sequence with this property is connected~\cite{Ostr66}. 

  To prove (b), consider a Fritz-John point satisfying MFCQ.
  Towards a contradiction, suppose $y^f=0$ in~\eqref{eq.FritzJohn}.
  Then there exists a nonzero, nonnegative linear combination of active constraint gradients equal to the zero vector. 
 The inner product of this linear combination with a direction $\mathbf{d}$ witnessing the MFCQ definition is strictly negative, a contradiction. Thus we must have $y^f>0$ for the Fritz-John point. 
We can then choose $\lambda=y^c/y^f$, where $y^f>0$ and $y^c$ are taken from the Fritz-John definition and observe that $\lambda$ witnesses the definition of a KKT point. 
The result follows from \Cref{thm.main}.
\end{proof}

\section{Active-set identification}\label{sec.identification}
In this section, we will ultimately demonstrate that the mechanism in \eqref{eq:active_subproblem} asymptotically identifies the correct active set at a limit point of $\{\xb_k\}$. 
Throughout this section, \Cref{assm:smoothness,assm:interpolation,assm:bounded_hessian} will hold.

Denote the Lagrangian $L(\xb,\lambda):=f(\xb)+\sum_{i\in P}\lambda_i c_i(\xb)$.
For the results in this section, we will need stronger assumptions than \Cref{assm:mfcq}. 
\begin{assumption}\label[assumption]{assm:nondegeneracy}
  There is a limit point $\bar\xb$ at which the following hold:
  \begin{enumerate}
    \item[(a)] \textbf{LICQ:} the gradients $\{\nabla c_i(\bar\xb):i\in\mathcal A(\bar\xb)\}$ are linearly independent. Provided no $\xb_k$ is a Fritz-John point, \Cref{thm.main} implies $\bar\xb$ a Fritz-John point; hence, LICQ automatically implies $\bar\xb$ is a KKT point with unique multiplier $\bar\lambda$.
    \item[(b)] \textbf{Strict complementarity:} $\bar\lambda_i>0$ for all $i\in\mathcal A(\bar\xb)$.
    \item[(c)] \textbf{Second-order sufficiency:} $f$ and $c_i$ are twice continuously differentiable in an open neighborhood of $\bar\xb$, and
    \[
      \db^T\nabla^2_{\xb\xb}L(\bar\xb,\bar\lambda)\db>0
      \text{ for every }\db\neq0\text{ such that }
      \nabla c_i(\bar\xb)^T\db=0\text{ for all } i\in\mathcal A(\bar\xb).
    \]
  \end{enumerate}
\end{assumption}
The conditions of \Cref{assm:nondegeneracy} make $\bar\xb$ a strict local minimizer. 
As shorthand, denote $\bar{\mathcal A}:=\mathcal A(\bar\xb)$ and $J(\xb):=[\nabla c_i(\xb)^T]_{i\in\bar{\mathcal A}}$. 
In expressions involving $J$, multipliers are restricted to $\bar{\mathcal A}$ and multipliers corresponding to $P\setminus\bar{\mathcal A}$ are interpreted as 0.

We first demonstrate an error bound for the exact criticality measure, and show that $\bar\xb$ is isolated. 
\begin{lemma}\label{lem.error_bound}\label{lem.isolated}
  Under \Cref{assm:nondegeneracy}, there exist $c_\chi,\epsilon_\chi>0$ such that
  \begin{equation}\label{eq.error_bound}
    \chi(\xb,\bar{\mathcal A})\geq c_\chi\|\xb-\bar\xb\|_\infty
    \quad\text{for every feasible }\xb\in\mathcal B(\bar\xb,\epsilon_\chi).
  \end{equation}
  The point $\bar\xb$ is isolated among feasible Fritz-John points. If also \Cref{assm:bounded_levelset} holds and no iterate is a Fritz-John point, then $\xb_k\to\bar\xb$.
\end{lemma}
\begin{proof}
Consider the equivalent definition of $\chi(\xb,\mathcal{A})$ in~\eqref{eq.dual_chi}.
With $\mathcal{A}=\bar{\mathcal A}$,~\eqref{eq.dual_chi} has minimum value zero and unique minimizer attained at
  $\bar y=(1,\bar\lambda)/(1+\mathbf1^T\bar\lambda)$.
  By stationarity and the LICQ, $y^c=y^f\bar\lambda$; the value of $y^f$ is determined by the constraints of $\Sigma_{\bar{\mathcal{A}}}$.
   By continuity and compactness of $\Sigma_{\bar{\mathcal A}}$, the dual minimizers $y(\xb)$ of~\eqref{eq.dual_chi} tends to $\bar y$ as $\xb\to\bar\xb$. Indeed, any limit of such minimizers minimizes the dual objective at $\bar\xb$ and hence equals $\bar y$.

  Thus, there exist $\alpha,\beta>0$ such that $y^f(\xb)\geq\alpha$ and
  $\lambda(\xb):=y^c(\xb)/y^f(\xb)\to\bar\lambda$, with $\lambda_i(\xb)\geq\beta$ for $i\in\bar{\mathcal A}$.
   For feasible $\xb$,~\eqref{eq.dual_chi} thus yields
  \begin{equation}\label{eq.chi_residual}
    \begin{aligned}
      \chi(\xb,\bar{\mathcal A})
      &=y^f(\xb)\big(\|\nabla f(\xb)+J(\xb)^T\lambda(\xb)\|_1
                +\lambda(\xb)^T[-c_{\bar{\mathcal A}}(\xb)]\big)\\
      &\geq\alpha\min\{1,\beta\}\,\psi(\xb,\lambda(\xb)),
    \end{aligned}
  \end{equation}
  where
  $\psi(\xb,\lambda):=\|\nabla f(\xb)+\sum_{i\in P}\lambda_i\nabla c_i(\xb)\|_1
  +\|\min\{\lambda,-c(\xb)\}\|_1$,
  with the minimum taken componentwise (recall our convention that $\lambda_i=0$ for any $i\in P\setminus\bar{\mathcal A}$). 
  If $\bar{\mathcal A}=\emptyset$, take $y^f=1$, $\lambda=\mathbf{0}$ and $\beta=1$.
A well-known local KKT error bound, see e.g.~\cite[Theorem~2.1]{OberWrig06} or~\cite[Theorem~3.1]{CurtRobiZebi25}, applies under \Cref{assm:nondegeneracy} and gives
  $\|\xb-\bar\xb\|_\infty\leq C\psi(\xb,\lambda(\xb))$ near $(\bar\xb,\bar\lambda)$.
  Together with~\eqref{eq.chi_residual}, this proves~\eqref{eq.error_bound}.

  Now, for each $i\notin\bar{\mathcal A}$, $c_i(\bar\xb)<0$. 
  By continuity of $c_i$, choose $0<\epsilon\leq\epsilon_\chi$ such that these constraints are uniformly inactive in $\mathcal B(\bar\xb,\epsilon)$. Hence $\mathcal A(\xb)\subseteq\bar{\mathcal A}$ for all $\xb\in\mathcal{B}(\bar\xb,\epsilon)$.
  If a feasible $\xb\in\mathcal{B}(\bar\xb,\epsilon)$ is a Fritz-John point, choose normalized multipliers $(y^f,y^c)$ satisfying~\eqref{eq.FritzJohn}. Complementarity gives $y_i^c=0$ for all $i\in P\setminus\mathcal A(\xb)$.
  Retaining $y^f$ and only the constraint multipliers indexed by $\bar{\mathcal A}$ therefore gives $(y^f,(y_i^c)_{i\in\bar{\mathcal A}})\in\Sigma_{\bar{\mathcal A}}$ with zero objective in~\eqref{eq.dual_chi}. Since that objective is nonnegative at feasible points, $\chi(\xb,\bar{\mathcal A})=0$, and~\eqref{eq.error_bound} yields
  \[
    0=\chi(\xb,\bar{\mathcal A})\geq c_\chi\|\xb-\bar\xb\|_\infty.
  \]
  Thus $\xb=\bar\xb$, proving isolation among feasible Fritz-John points.

  If \Cref{assm:bounded_levelset} holds and no $\xb_k$ is a Fritz-John point, then, since $\bar\xb\in\mathcal X$ and every element of $\mathcal X$ is a Fritz-John point by \Cref{cor.convergence}(a), the isolation just proved makes $\bar\xb$ isolated in $\mathcal X$, and \Cref{cor.convergence}(a) gives $\mathcal X=\{\bar\xb\}$ and $\xb_k\to\bar\xb$.
\end{proof}

We next identify the range of radii where the active-set estimate is exact. Define
\begin{equation*}%
  \nu:=\tfrac12\min_{i\notin\bar{\mathcal A}}[-c_i(\bar\xb)],
  \qquad G:=\sqrt n\max_{i\in P}(\rho_{\nabla c_i}+\kappa_{\nabla c_i}),
\end{equation*}
where we interpret $\nu=\infty$ if $\bar{\mathcal A}=P$. 
Further define constants
\begin{equation*}%
    C_1:=\max_{i\in\bar{\mathcal A}}
          \frac{2\sqrt n\,\rho_{\nabla c_i}}{\|\nabla c_i(\bar\xb)\|_1}, \quad
    C_2:=\min\left\{1,\frac{\nu}{2G},
          \min_{i\in\bar{\mathcal A}}\frac{\|\nabla c_i(\bar\xb)\|_1}
                    {4\sqrt n\,\kappa_{\nabla c_i}}\right\}>0.
\end{equation*}
If $\bar{\mathcal A}=\emptyset$, we interpret $C_1=0$ and omit the third term defining $C_2$.

\begin{lemma}\label{lem.window}
  Under \Cref{assm:nondegeneracy}, there exists $\epsilon_1>0$ such that, at any $\xb_k$ satisfying $\|\xb_k-\bar\xb\|_\infty\leq\epsilon_1$, the estimate~\eqref{eq.Ak_explicit} satisfies $\Ak=\bar{\mathcal A}$ whenever
  \begin{equation}\label{eq.window}
    C_1\|\xb_k-\bar\xb\|_\infty\leq\Delta_k^\sigma\leq C_2.
  \end{equation}
  The trial point then satisfies $c_i(\xb_k + \stepb_k) < 0$ for all $i\in P\setminus\bar{\mathcal A}$.
\end{lemma}
\begin{proof}
  Choose $\epsilon_1$ such that for $\{\xb: \|\xb-\bar\xb\|_{\infty}\leq \epsilon_1\}$, we have $c_i(\xb)\leq-\nu$ for $i\notin\bar{\mathcal A}$ and
  $\|\nabla c_i(\xb)\|_1\geq\tfrac34\|\nabla c_i(\bar\xb)\|_1$ for $i\in\bar{\mathcal A}$.
  Since $\Delta_k\leq\Delta_k^\sigma\leq C_2\leq1$, full linearity gives $\|g_k^{c_i}\|_1\leq G$ for every $i$ and
  $\|g_k^{c_i}\|_1\geq\tfrac12\|\nabla c_i(\bar\xb)\|_1$ for $i\in\bar{\mathcal A}$.
  Writing $d:=\|\xb_k-\bar\xb\|_\infty$, it follows that
  \[
    c_i(\xb_k)+\Delta_k^\sigma\|g_k^{c_i}\|_1
    \begin{cases}
      \leq-\nu+GC_2\leq-\nu/2<0,&i\notin\bar{\mathcal A},\\
      \geq-\sqrt n\rho_{\nabla c_i}d
         +\tfrac12\Delta_k^\sigma\|\nabla c_i(\bar\xb)\|_1\geq0,&i\in\bar{\mathcal A},
    \end{cases}
  \]
  proving exact identification. For $i\notin\bar{\mathcal A}$ and $\|\stepb\|_\infty\leq\Delta_k$, bounded gradients also give
  $c_i(\xb_k+\stepb)\leq-\nu+\sqrt n\rho_{\nabla c_i}\Delta_k\leq-\nu/2<0$.
\end{proof}

We will next show that the criticality loop beginning in \Algline{alg.main_alg_chi}{step.criticality} suffices to yield $\Delta_k$ satisfying~\eqref{eq.window}.
Towards this end, define constants
\begin{equation*}%
  R:=C_2^{1/\sigma},\qquad
  \overline\mu:=\frac{R}{\sqrt n(\rho_{\nabla f}+\kappa_{\nabla f}R)}.
\end{equation*}

\begin{lemma}\label{lem.landing}
  Let \Cref{assm:nondegeneracy} hold and let $0<\mu\leq\overline\mu$. 
  There exists $\epsilon_2>0$ such that, if $\xb_k$ satisfies
  $0<\|\xb_k-\bar\xb\|_\infty\leq\epsilon_2$ and the criticality loop is entered with $\Delta_k\geq \underline\Delta$, then the criticality loop terminates finitely at a radius satisfying~\eqref{eq.window}.
\end{lemma}
\begin{proof}
 Define $d:=\|\xb_k-\bar\xb\|_\infty$, $\chi:=\chi(\xb_k,\bar{\mathcal A})$, 
    $r:=(C_1d)^{1/\sigma}$, and 
    $T:=\min\{\chi/(2\kappa_\chi),\mu\chi/2,R\}.$
  By \Cref{lem.error_bound}, $\chi\geq c_\chi d$, so $T\geq a_0d$ for some $a_0>0$ and all sufficiently small $d>0$. Since $\sigma<1$, $r=o(d)$.
  Thus, we can choose $\epsilon_2\leq\min\{\epsilon_1,\epsilon_\chi\}$ small enough that
   $r\leq\min\{\underline\Delta,\gammad T\}$
 for all $d\in(0,\epsilon_2]$.

  At any radius $\Delta>R$, monotonicity and full linearity give
  \[
    \mu\chi_k(\xb_k,\Ak)\leq\mu\|g_k^f\|_1
    \leq\frac{R(\rho_{\nabla f}+\kappa_{\nabla f}\Delta)}
                  {\rho_{\nabla f}+\kappa_{\nabla f}R}<\Delta.
  \]
  Thus the criticality loop cannot exit with a $\Delta_k$ greater than $R$. 
  If the criticality loop exits before reaching a radius at most $T$, then $\Delta_k\in(T,R]\subseteq[r,R]$. Otherwise, the first radius $\Delta\leq T$ satisfies $\Delta\geq r$: $\Delta$ is either the radius that started the criticality loop or else its predecessor exceeds $T$ and $\Delta>\gammad T\geq r$. Therefore \Cref{lem.window} gives $\Ak=\bar{\mathcal A}$, and~\eqref{eq.chi_error} yields
  \[
    \chi_k(\xb_k,\Ak)\geq\chi-\kappa_\chi\Delta\geq\chi/2,
    \qquad \Delta\leq T\leq\mu\chi/2\leq\mu\chi_k(\xb_k,\Ak).
  \]
  The criticality loop must exit at this radius $\Delta$. 
  In either case, the criticality loop terminates with $\Delta_k\in[r,R]$, which is precisely~\eqref{eq.window}.
\end{proof}

\begin{theorem}[Active-set identification]\label{thm.identification}
  Suppose \Cref{assm:smoothness,assm:interpolation,assm:bounded_hessian,assm:bounded_below,assm:bounded_levelset,assm:mfcq,assm:nondegeneracy} hold, $0<\mu\leq\overline\mu$, and no $\xb_k$ is a Fritz-John point.
   Then $\xb_k\to\bar\xb$. Moreover, for sufficiently large $k$ with criticality loop entered with radius at least $\underline\Delta$,
  \begin{enumerate}
    \item[(a)] the criticality loop terminates with $\Ak=\bar{\mathcal A}$ and~\eqref{eq.window};
    \item[(b)] $\Ak$ is not augmented, so \eqref{eq:step_subproblem} in the $k$th iteration uses $\Ak=\bar{\mathcal A}$.
  \end{enumerate}
  These conclusions apply, in particular, to all sufficiently large $k$ immediately following an accepted step, and hence, to infinitely many iterations.
\end{theorem}
\begin{proof}
  By \Cref{lem.error_bound}, $\xb_k\to\bar\xb$. There is no $\xb_k$ such that $\xb_k=\bar\xb$, since it is a KKT point and therefore a Fritz-John point. Thus \Cref{lem.landing} applies to all sufficiently large $k$ with radius going into the criticality loop bounded below by $\underline\Delta$, proving (a). 
  By \Cref{lem.window}, no constraint in $P\setminus\bar{\mathcal A}$ can be violated at any trial point, so \Algline{alg.main_alg_chi}{line:expand_active_set} is never reached, proving (b).
  Finally, \Cref{thm.main} gives infinitely many accepted steps, each of which resets the radius going into the criticality loop to at least $\underline\Delta$.
\end{proof}

\section{Numerical Experiments} \label{sec.experiments}

In this section, we test the two issues motivating the design of \Cref{alg.main_alg_chi}, together with the identification result of \Cref{thm.identification}. 
Accordingly, we ask three questions. 
\begin{enumerate}
\item Does the FDM subproblem~\eqref{eq:step_subproblem} produce feasible trial points more reliably than the standard TR-SQP subproblem~\eqref{eq:ideal_sqp} when both are embedded in the same feasible-incumbent framework? 
\item How much of the cost of approximating the constraint Jacobian is saved by the active-set estimate $\Ak$ and by minimally updating models of constraints outside $\Ak$, and at what cost in progress? 
\item Is the optimal active set eventually identified in practice, as \Cref{thm.identification} predicts? 
\end{enumerate}
\subsection{Implementation}\label{sec.impl}

We implemented \Cref{alg.main_alg_chi} in Python.
The details that the analysis leaves open are resolved as follows.

\paragraph{Objective and constraint models} 
The linear models in \Cref{alg.main_alg_chi} are obtained from the model-building routine \texttt{formquad} of \texttt{POUNDERS}\cite{ibcdfo}, applied to each function $(f(\xb), c_1(\xb), \dots, c_p(\xb)$ individually.
Essentially, \texttt{formquad} will search the history of past evaluations of a function for $n+1$ affinely independent points including, and centered around, $\xb_k$. 
Additional points can be added as necessary, and if geometrically appropriate in a precise sense, to fit a minimum Frobenius norm quadratic correction. 
Details of \texttt{formquad} are provided in~\cite[Section 4]{SW08}.
 The resulting models are fully linear on $\mathcal{B}(\xb_k,\Delta_k)$ in the sense of \Cref{assm:interpolation}, see \cite[Theorem 4.1]{SW08}. 
 In our algorithm, and in accordance with our theoretical results, we only employ linear models as $m^f$ and $m^{c_i}$ for each $i=1,\dots,p$. 
 That is, we use the gradient of the model for $f$ returned by \texttt{formquad} as $g_k^f$ and the gradient of the model of each $c_i$ as returned by \texttt{formquad} as $g_k^{c_i}$.
 For $H_k$ in \cref{eq:step_subproblem}, we clip 
the eigenvalues of the model Hessian for $f$ returned by \texttt{formquad} to lie between $[10^{-6},10^{8}]$ and use that matrix, so that \Cref{assm:bounded_hessian} holds.
We do not make use of any Hessian term returned by \texttt{formquad} for any constraint models.
Note that \texttt{formquad} draws on the history of past evaluations of $f$
without regard to feasibility, so $m^f$ may be built using objective values at
infeasible points, even though the incumbents $\xb_k$ themselves remain
feasible. 
It is this precisely this shortcoming that would need to be addressed in extending our method to handle unrelaxable constraints. 

\paragraph{Lazy constraint models}
In the theory all $p$ constraint models are made fully linear at every iteration of \Cref{alg.main_alg_chi}, see \Cref{assm:interpolation}. To limit the number of constraint evaluations, we experiment with loosening this. Because every incumbent is tested for feasibility in \Algline{alg.main_alg_chi}{line:eval_ci}, the values $c_i(\xb_k)$ are known exactly; only the gradients are in question, and \eqref{eq.Ak_explicit} uses them only through $\|g^{c_i}_k\|_1$, weighed against that exact slack. A gradient $g^{c_i}_{\ell}$ from an earlier iteration $\ell < k$ may therefore still place $c_i$ on the correct side of the test. 

The lazy variant interleaves \Algline{alg.main_alg_chi}{step.models} and \Algline{alg.main_alg_chi}{step.active_set}. We evaluate \eqref{eq.Ak_explicit} with the exact values and the most recent gradients; build the models of $f$, of the constraints in that preliminary estimate, of those in $\mathcal{A}_{k-1}$, and of any constraint not rebuilt in the last $r_{\rm refresh}$ iterations (default $10$), evaluating geometry points for these components alone; and evaluate \eqref{eq.Ak_explicit} again with the fresh gradients. A constraint entering the estimate only at this point has its model built as well, after which the estimate is formed once more. Constraints that stay outside keep their gradients and are evaluated at no geometry point of this iteration, so a geometry point costs $1+|\Ak|$ evaluations rather than $p+1$. 

We refer to the faithful representation of \Cref{alg.main_alg_chi} and this lazy variant as \emph{all-models} and \emph{lazy}, respectively; \Cref{sec.exp_cost} quantifies the difference.

\paragraph{Subproblems} 
The step subproblem \eqref{eq:step_subproblem} and the criticality measure \eqref{eq.chik_def} (obtained from \eqref{eq:step_subproblem} by omitting the quadratic term and using the unit ball) are convex quadratic (respectively linear) programs in the variables $(\stepb,\tau)$ with $1 + |\Ak| + 2n$ linear inequality constraints. 
We solve these subproblems with the interior-point solver Clarabel \cite{goulart2024clarabel}, with all tolerances set to $10^{-10}$. 
A step $\stepb_k$ satisfying $\|\stepb_k\|_\infty \leq 10^{-8}\Delta_k$ is treated as $\stepb_k = 0$.

\paragraph{Constraint evaluation at trial points}
In \Algline{alg.main_alg_chi}{line:eval_ci} the constraints are evaluated one at a time and the loop exits at the first violation; any value that is not a floating-point number (for instance, \texttt{NaN} or infinite-valued) is treated as a violation. 
We evaluate the constraints in $\Ak$ before we evaluate the constraints in $P\setminus \Ak$, since the constraints in $\Ak$ are empirically more likely to be violated.

\paragraph{Parameters and termination}
Unless stated otherwise, $\Delta_0 = 1$, the parameters of \Cref{alg.main_alg_chi} are set as $\underline\Delta = 10^{-3}$, $\mu = 1$, $\sigma = 0.5$, $\gammad = 0.5$, $\gammai = 2$, $\eta_1 = 0.2$, and $\eta_2 = 0.5$. 
Each run is given a budget of $200(n+1)(p+1)$ oracle queries; we will discuss accounting of oracle queries further in \Cref{sec.exp_measures}.
We also terminate a run if $\Delta_k < 10^{-10}$, or if, simultaneously, $\chi_k(\xb_k,\Ak) \leq 10^{-8}$ and $\Delta_k \leq 10^{-8}$ upon exiting the criticality loop. 
While a proper stopping criterion should also depend on a guarantee that the active set was correctly identified, the condition $\mu \leq \overline\mu$ of \Cref{thm.identification} cannot be checked in practice since $\overline\mu$ depends on unknown constants.
Nevertheless, the optimal active set is identified on most problems, as we will exhibit in \Cref{sec.exp_identification}.

\subsection{Test problems}\label{sec.exp_problems}
We draw test problems from the CUTEst collection~\cite{gould2015cutest} through the PyCUTEst interface~\cite{fowkes2022pycutest}. 
From all problems classified as ``regular" with general (linear, quadratic, or other) constraints and $n \leq 30$ variables, we keep those problems with inequality constraints only. 
We rewrite every constraint $c_\ell \leq C(\xb) \leq c_u$ with finite $c_\ell$ or $c_u$ as a single constraint $c_i(\xb) \leq 0$.
Because \Cref{alg.main_alg_chi} has no separate mechanism for simple bounds, each finite variable bound is treated the same way and counted in $p$; problems with $p > 100$ after this conversion are discarded.
We also remark here that an extension of \Cref{alg.main_alg_chi} that handles simple bound constraints would be straightforward to implement, but we do not want to distract from the main intentions of this paper. 

\Cref{alg.main_alg_chi} requires a feasible starting point $\xb_0$. In the applications that motivate this work, where constraints are unrelaxable, such a point is typically known in advance.
Finding a feasible starting point when none is known a priori is a separate and challenging problem that we do not address. 
For the CUTEst problems in our subset with an infeasible starting point, we replace the starting point by the nearest point satisfying $c_i(\xb) \leq -10^{-3}(1+|c_i(\xb_0)|)$ for all $i$, as computed by a derivative-based SQP solver. 
Six problems for which no such feasible starting point was found were discarded. 

\subsection{Solvers and performance measures}\label{sec.exp_measures}
We compare the following variants of our implementation of \Cref{alg.main_alg_chi}, designed to probe various features we discussed in \Cref{sec.impl}.

\begin{itemize}
    \item \solver{FDM-AS-full}: The most faithful direct representation of \Cref{alg.main_alg_chi} that was analyzed in this paper. Most importantly (and expensively), this variant will build/update all $p$ constraint models in every iteration.
    \item \solver{FDM-AS-lazy}: The method \solver{FDM-AS-full}, except we employ lazy model updates, as  discussed in \Cref{sec.impl}. 
    \item \solver{FDM-ALL}: The method \solver{FDM-AS-full}, except we set $\Ak = \{1,\dots,p\}$ in every instance of the step subproblem \eqref{eq:step_subproblem}. That is, we effectively disable active-set estimation. 
    \item \solver{SQP-ALL}: The method \solver{FDM-AS-full}, except we replace every instance of the step subproblem \eqref{eq:step_subproblem} with \eqref{eq:ideal_sqp}.
    That is, we employ an SQP subproblem, using the same linear constraint models and the same quadratic term $H_k$ that would have been supplied to \eqref{eq:step_subproblem}. 
    In this variant, whenever \eqref{eq:ideal_sqp} admits an infeasible step $\stepb_k$, or when the predicted objective decrease satisfies $m^f(\xb_k) - m^f(\xb_k+\stepb_k)<0$, the iteration is unsuccessful, $\Delta_k$ is reduced, and a new iteration begins. 
    \item \solver{SQP-AS-lazy}: This is the method \solver{SQP-ALL}, except we enable the lazy model updates employed in \solver{FDM-AS-lazy}.
    \item \solver{CONORBIT}: A method of Regis and Wild~\cite{regis2017conorbit}, which builds radial basis function models of the objective and constraints and, like our method, aims to maintain feasible incumbents. \solver{CONORBIT} queries $f$ and all $p$ constraints at every point it selects. We ran the \solver{CONORBIT} implementation provided to us by its developers in MATLAB R2026a. We used the author's default settings in their driver script.
\end{itemize}

\paragraph{Oracle cost}
For our experiments, we cast the solution of problem~\eqref{eq:opt_prob} as requiring access to $p+1$ distinct oracles: one oracle for the objective $f$ and one for each of the constraints $c_1,\dots,c_p$. 
Our budget unit for these experiments is one scalar query of any one of these oracles at any point $\xb\in\mathbb{R}^n$. 
In this cost model, a solver is charged only for the queries it makes. 
This method of accounting targets our second premise we cast in \Cref{sec.intro}, namely that a method for expensive simulation-based optimization should minimally evaluate any given oracle. 
This cost model assumes that all constraints are computationally independent. 
The criterion of \Algline{alg.main_alg_chi}{line:eval_ci}, which halts constraint evaluations upon the first violated constraint, together with the lazy models of \Cref{sec.impl}, which only updates constraint models indexed in $\Ak$, are both intentionally designed to be cost-effective according to such an accounting model. 

In presenting our results, we normalize total oracle queries, aggregated across all $p+1$ oracles, by $(n+1)(p+1)$.
This means that one unit under this normalization is equivalent to the cost of computing one forward-difference approximation of the Jacobian of the vector-valued function $(f,c_1,\dots,c_p)$ evaluated at an arbitrary point $\xb\in\mathbb{R}^n$.
 We also state budgets in this normalized metric: each run is given a budget of $200(n+1)(p+1)$ oracle queries, that is $200$ approximate Jacobians, after normalization.

\paragraph{Data profiles}
The data profiles we employ here deviate slightly from standard data profiles, as described by Mor\'e and Wild~\cite{more2009benchmarking}. 
Given our motivation of computational settings in which constraint violations may suggest physically meaningless results, we only wish to reward low objective values provided the corresponding 
approximate minimizer is feasible.
For a point $\xb$ visited by a run of a solver $s$, we denote by $t(\xb)$ the cumulative number of oracle queries the run of solver $s$ has issued at the moment the last of $f(\xb),c_1(\xb),\dots,c_p(\xb)$ is returned. 
This bookkeeping ensures that the solver $s$ cannot receive credit for identifying an approximate minimizer $\xb$ until the solver $s$ has verified that $\xb$ is indeed feasible. 

We now demonstrate how we use $t(\xb)$ in our data profile. 
Let $\mathcal{P}$ denote the test set described in \Cref{sec.exp_problems}.
For $j\in\mathcal{P}$, let $p_j$ be the number of constraints in the $j$th problem, and let $\xb_0^j$ denote the feasible starting point, common to all solvers. 
Let $n_j$ denote the number of variables for the $j$th problem. Fix $\varepsilon_\tau \in \{10^{-1},10^{-3},10^{-5},10^{-7}\}$ and let $f_L^j$ be the smallest objective value attained at a feasible point by any solver on problem $j$ within the budget. 
Any solver $s$ generates a sequence of oracle queries $\{(\xb, g)\}$, where $\xb$ is the queried point and $g\in\{f,c_1,\dots,c_p\}$ denotes the oracle queried. 
In post-processing, one can convert $\{(\xb,g)\}$ into a sequence $\{(t(\xb),f(\xb))\}$, where we denote $t(\xb)=\infty$ if not all $p+1$ distinct oracle calls were made at $\xb$ in the sequence $\{(\xb,g)\}$.
Then, we define 
\begin{equation}\label{eq.convtest}
  t_{j,s} := \min\Big\{ t(\xb) : \{(t(\xb),f(\xb))\}\ \text{generated by $s$ and }\ f(\xb) \leq f_L^j + \varepsilon_\tau\big(f(\xb_0^j)-f_L^j\big) \Big\}. 
\end{equation}
If $f(\xb) \leq f_L^j + \varepsilon_\tau\big(f(\xb_0^j)-f_L^j\big)$ is never satisfied, we define $t_{j,s}=\infty$ by convention. 
 The data profile for solver $s$ is then 
\begin{equation}\label{eq.dataprofile}
  d_s(\alpha) := \frac{1}{|\mathcal{P}|}
  \left| \left\{ j \in \mathcal{P} : \frac{t_{j,s}}{(p_j+1)(n_j+1)} \leq \alpha \right\} \right| ,
\end{equation}
or the fraction of the test set $\mathcal{P}$ that solver $s$ solves to accuracy $\varepsilon_\tau$ within $\alpha$ normalized evaluations. 
We compare multiple solvers on one profile per value of $\varepsilon_\tau$. 

\subsection{Feasibility of trial points}\label{sec.exp_feasibility}

\Cref{fig.feasibility} addresses our first question. 
For each run, we record the fraction of trial points $\xb_k + \stepb_k$ that satisfy $c(\xb_k+\stepb_k) \leq 0$, that is, that reach \Algline{alg.main_alg_chi}{line:iviol_zero}. 
Given our motivation, this fraction is of practical concern to us;
an infeasible trial point is always rejected.
By the logic of \Cref{alg.main_alg_chi}, this rejection either results in the violated constraint being added to $\Ak$ or a reduction in $\Delta_k$; both outcomes generally entail building or improving constraint models at the cost of constraint evaluations. 
We stress that this fraction is not a measure of progress;
a method may produce a larger fraction of infeasible trial points and still reach a solution sooner.
Overall performance is instead compared via data profiles in \Cref{sec.exp_cost}.

With the FDM subproblem \eqref{eq:step_subproblem}, the median percentage of feasible trial points generated by \solver{FDM-ALL} over all the test problems is $98.7\%$; similarly, $98.4\%$ for \solver{FDM-AS-full} and $97.4\%$ for \solver{FDM-AS-lazy}.
With the SQP subproblem \eqref{eq:ideal_sqp}, the median percentage is only $7.9\%$ for \solver{SQP-ALL} and $0.9\%$ for \solver{SQP-AS-lazy}. 
Fewer than half of the trial points are feasible on $24$--$26\%$ of the problems for the FDM variants, compared with $73\%$ for \solver{SQP-ALL} and $87\%$ for \solver{SQP-AS-lazy}. 
We argue this is because
the step generated by the SQP subproblem \eqref{eq:ideal_sqp} must land on the boundary of the linearized constraints; the curvature of the constraints often makes such steps infeasible.
On the other hand the step generated by the FDM subproblem \eqref{eq:step_subproblem} generally lies in the interior of the linearized feasible region (\Cref{fig.intro}(a)).

\begin{figure}
    \centering
    \includegraphics[width=0.7\linewidth]{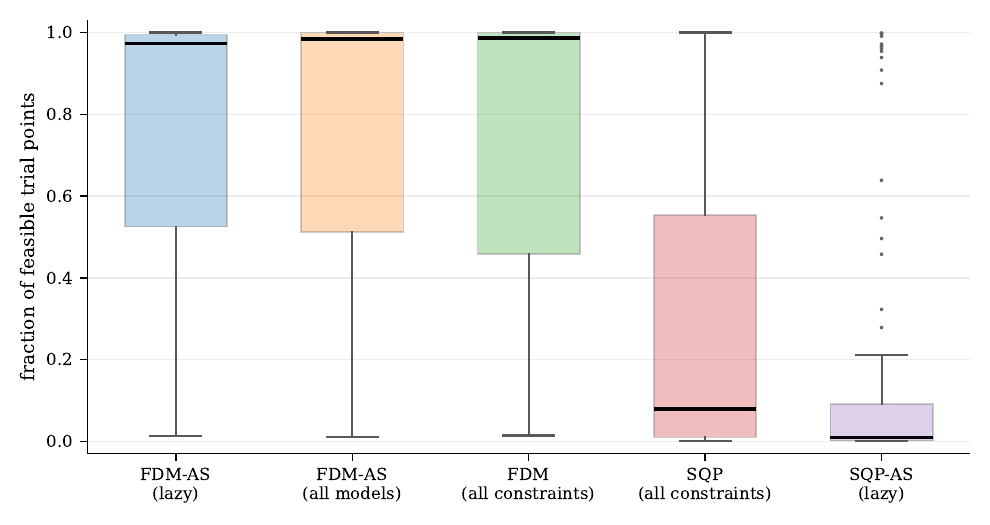}
    \caption{Fraction of trial points $\xb_k+\stepb_k$ satisfying $c(\xb_k+\stepb_k)\leq 0$, over the $127$ test problems.  The step subproblem \eqref{eq:step_subproblem} produces feasible trial points on the large majority of iterations; the SQP-like subproblem \eqref{eq:ideal_sqp} does not.}
    \label{fig.feasibility}
\end{figure}

\subsection{Oracle cost}\label{sec.exp_cost}

\Cref{fig.profiles} addresses the second question outlined in this section. 
Over the $127$ problems in $\mathcal{P}$, \solver{FDM-AS-lazy} evaluates a median of $87.5\%$ of the
$p$ constraints at each evaluated point, compared to a median of $99.6\%$ for
\solver{FDM-AS-full} and $99.9\%$ for \solver{FDM-ALL}.
As a reminder, the latter two methods rebuild models at every iteration, while the former does not.
We remark that this experiment also illustrates that, empirically, it is not simply the estimation of $\Ak$ that yields saving in total oracle queries. 
Active-set estimation reduces the size
of the subproblem \eqref{eq:step_subproblem}, not the number of models built or updated; this explains why
\solver{FDM-AS-full} and \solver{FDM-ALL} are nearly indistinguishable in this metric. 

The benefit of lazy models is most visible early in the run. 
In \Cref{fig.profiles}, \solver{FDM-AS-lazy} 
has a very slight lead over the other FDM variants for roughly the first normalized evaluations
at the tighter accuracies: At $\varepsilon_\tau=10^{-5}$ and $\alpha=10$ it
solves $18\%$ of the problems, versus $12\%$ for \solver{FDM-AS-full} and $9\%$
for \solver{FDM-ALL}.
Over the rest of the budget at tigher $\varepsilon_\tau$ levels, the three FDM essentially indistinguishable. 
The FDM variants solve far more problems than the SQP variants at every
accuracy once $\alpha$ exceeds about $10$.

\begin{figure}[h!]
    \centering
    \includegraphics[width=0.85\linewidth]{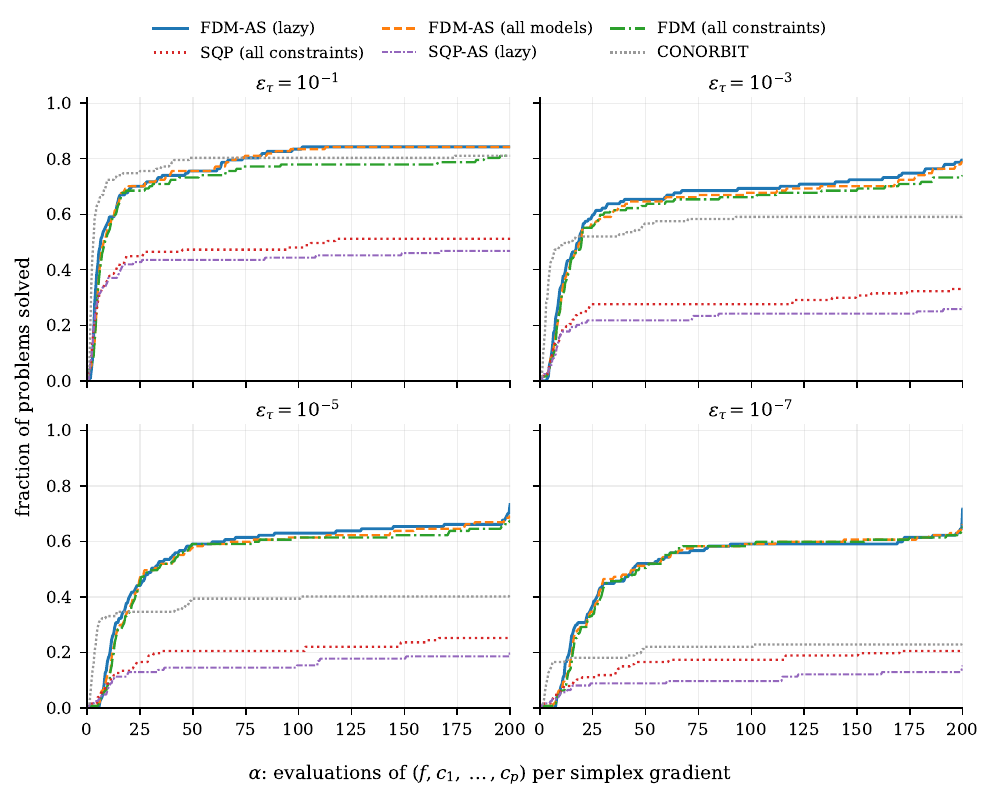}
    \caption{Data profiles \eqref{eq.dataprofile} at four accuracies.  One unit on
      the horizontal axis is one forward-difference approximation of
      $(f,c_1,\dots,c_p)$, so the budget is reached at $\alpha = 200$.  An
      objective value is credited only at a point where every constraint has been
      queried and found satisfied.}
    \label{fig.profiles}
\end{figure}

\subsection{Active-set identification}\label{sec.exp_identification}

For the purpose of probing the third question outlined in the start of this section, concerning the utility of active set identification, we must identify a reference solution $\bar{\xb}$ for each problem in $\mathcal{P}$ so that we may compare an estimate $\Ak$ to the true active set $\mathcal{A}(\bar{x})$. 
Of the problems in $\mathcal{P}$ identified in \Cref{sec.exp_problems}, we computed a reference solution $\bar{\xb}$ by running the solvers \texttt{SLSQP} and \texttt{trust-constr} as implemented in SciPy~\cite{scipy}.
These solvers were supplied with exact Jacobians and initialized five times, once with the starting point $\xb_0$ and four times with random perturbations of the starting point.
We record the reference point $\bar{\xb}$ as the least-objective feasible point returned by the solver, among all five runs. 
Any problem for which a feasible result could not be obtained was discarded.
The resulting problem set contains 127 problems with $n \in [2,30]$ and $p \in [1,90]$.
Once a reference point $\bar{\xb}$ is computed,
the reference active set $\bar{\mathcal{A}}$ is derived from $\bar{\xb}$ by indexing all constraints satisfying $c_i(\bar{\xb}) \geq -10^{-6}(1+|c_i(\bar{\xb})|)$.
We note that strict complementarity, which \Cref{thm.identification} requires, holds at 114 of the 127 reference points. 

Note that, due to our procedure for generating reference points $\bar{\xb}$, the reference point is independent of any of the solvers that we compare.
Moreover, we note that $\bar{\xb}$ is only provably a stationary point; we have no guarantee that any $\bar{\xb}$ is globally optimal.

With this setup, \Cref{fig.identification} provides evidence to answer our third question.
Recall that \Cref{thm.identification} states, near a nondegenerate limit point $\bar{\xb}$ if the algorithmic parameter $\mu>0$ is chosen sufficiently small, 
then the estimate $\Ak$ produced at the end of the criticality loop satisfies $\Ak=\bar{\mathcal{A}}$ for all $k$ such that the $(k-1)$th iteration resulted in an accepted step. 
\Cref{thm.identification} specifically concerns the active set at the limit point $\bar\xb$ of $\{\xb_k\}$.
If $\{\xb_k\}$ converges to a limit point that does coincide with a reference point that we identified, or if within a finite budget, $\{\xb_k\}$ does not seem to converge to any limit point, then \Cref{thm.identification} need not apply.
On only $77$ of the $114$ problems for which we identified a reference point satisfying strict complementarity did \solver{FDM-AS-lazy} end at an incumbent $\xb_k$ that was
within $10^{-2}$ (in a relative $\|\cdot\|_\infty$ norm) of the reference point $\bar\xb$.
However, of the $77$ runs satisfying the criteria of \Cref{thm.identification}, noting that we do not know or estimate $\bar\mu$, the estimate $\Ak=\mathcal{A}(\bar\xb)$ throughout the tail of the run, as desired. More specifically, on $74$ of these $77$ problems, $\Ak = \bar{\mathcal{A}}$ at every iteration of the last quarter of the run. 
\Cref{fig.identification} shows four such runs. 
Each panel of \Cref{fig.identification} plots differences between the sets $\Ak$ and $\bar{\mathcal{A}}$.
We plot the cardinality $|\Ak \setminus \bar{\mathcal{A}}|$ above the center axis of each panel, and $|\bar{\mathcal{A}} \setminus \Ak|$ below it. 
Both metrics of overestimation and underestimation, respectively, vanish exactly when $\Ak = \bar{\mathcal{A}}$.
The vertical dashed line in each panel marks the iteration $k$ after which $\Ak=\bar{\mathcal{A}}$ consistently. 

In general, of the $37$ runs that do not satisfy the criteria of \Cref{thm.identification}, the active set estimate $\Ak$ does not settle at $\mathcal{A}(\bar\xb)$.

\begin{figure}
    \centering
    \includegraphics[width=0.7\linewidth]{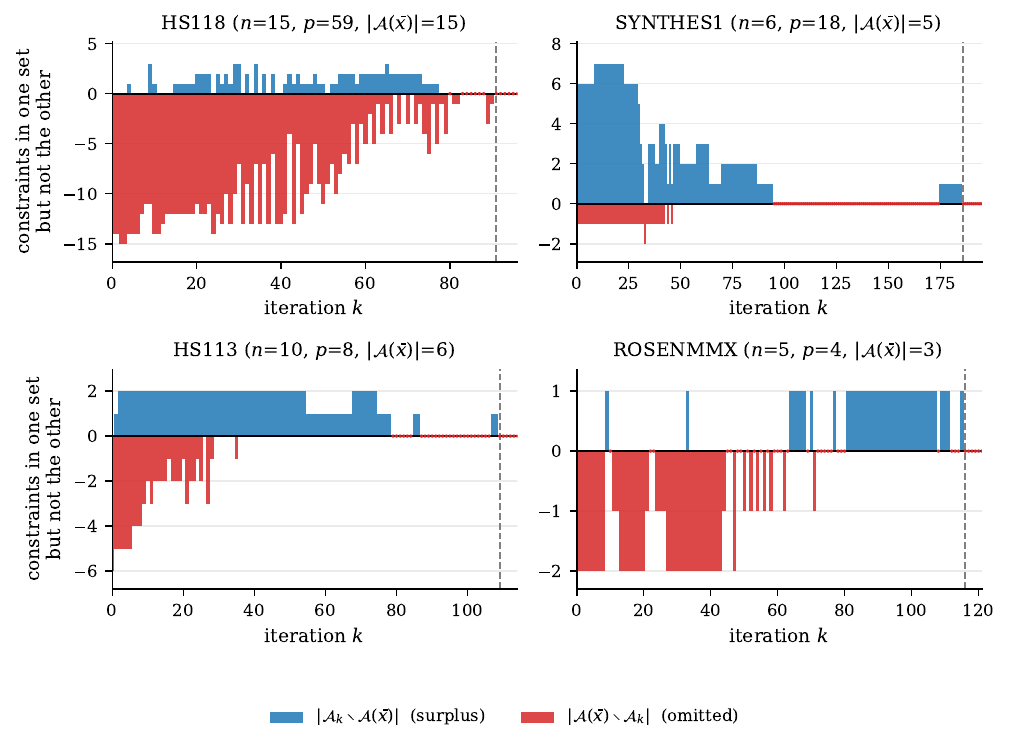}
    \caption{Set differences between the estimate $\Ak$ and the active set $\bar{\mathcal{A}}$ at the reference solution, on four problems.  Above the axis, constraints in $\Ak$ that are not active at $\bar{\xb}$; below it, active constraints that $\Ak$ omits. Both vanish exactly when $\Ak = \bar{\mathcal{A}}$.}
    \label{fig.identification}
\end{figure}

\section{Conclusion}
In this paper, we proposed a model-based trust-region method for derivative-free optimization with inequality constraints that deliberately maintains feasible incumbents and aims to avoid evaluating likely inactive constraints. 
Our method computes steps using a subproblem suitable for a feasible direction method and imposes only the constraints in an active-set estimate formed on a ball larger than the trust region.
We demonstrated that every limit point of the incumbent solutions generated by our methods is a Fritz-John point, and is moreover a KKT point under the MFCQ.
Under standard nondegeneracy conditions, our active-set estimate identifies the constraints active at the solution. 
On appropriately selected CUTEst problems, our method produced feasible trial points far more often than a variant of our method that, as an ablation, only replaced our subproblem with an SQP subproblem.
Moreover, a variant of our method that only builds or updates constraint models corresponding to our active-set estimate reduced the number of constraint evaluations per point and solved problems sooner than a variant of our method that consistently built or updated all constraint models; with a large budget, the two variants solved nearly the same fraction of problems. 
Our active set estimate also identified the true active set of a reference solution on most problems with runs that converge to a reference solution, as predicted by our theory. 

Several questions remain open. 
First, our analysis assumes fully linear models of all constraints at every iteration, whereas the lazy variant that performed best in our experiments does not; extending the theory to models that are accurate only for the constraints in $\Ak$ would be a natural next step. 
Second, and of critical empirical importance, although the incumbents of our method are always feasible, our method presently builds or updates constraint models using interpolation at potentially infeasible values. 
Building constraint models using only feasible points would make the method fully applicable to unrelaxable constraints; a model-building method resembling the method of Hough and Roberts~\cite{hough2022model}, but adapted to general simulation-based constraints (as opposed to convex algebraic constraints) is a likely avenue for success in this direction.

\section*{Software availability}
The code used to produce the numerical results in this paper is available at
\url{https://github.com/POptUS/LMZ-2026-code}. 
We do not include the source code of CONORBIT~\cite{regis2017conorbit}.
Readers wishing to reproduce our full results should contact the authors of
CONORBIT to obtain their code.

\bibliographystyle{siamplain}
\bibliography{refs_siopt}
\end{document}